\documentclass[12pt]{amsart}
\usepackage[english]{babel}
\usepackage[utf8]{inputenc}
\usepackage[T1]{fontenc}
\usepackage{hyperref}
\usepackage{dsfont}
\usepackage{color}

\usepackage{stmaryrd}
\usepackage{cite}

\usepackage{pgf,tikz,pgfplots}
\usepackage{mathrsfs}
\usetikzlibrary{arrows}

\usepackage{geometry}
\allowdisplaybreaks

\usepackage{amsmath}
\usepackage{amssymb}
\usepackage{amsthm}
\usepackage{amsfonts}

\newcommand{\R}{\mathbb{R}}
\newcommand{\N}{\mathbb{N}}

\renewcommand{\div}{\text{div }}
\newcommand{\curl}{\text{curl }}

\newcommand{\tu}{\tilde{u}}
\newcommand{\tv}{\tilde{v}}
\newcommand{\hu}{\hat{u}}
\newcommand{\omein}{\omega_{\mathrm{in}}}
\newcommand{\tomega}{\tilde{\omega}}
\newcommand{\homega}{\hat{\omega}}
\newcommand{\tomein}{\tilde{\omega} _{\mathrm{in}}}
\newcommand{\tomeout}{\tilde{\omega} _{\mathrm{out}}}
\newcommand{\omeout}{\omega_{\mathrm{out}}}
\newcommand{\Go}{\Gamma_{\mathrm{out}}}
\newcommand{\pn}{\partial_{\mathbf{n}}}
\newcommand{\Gi}{\Gamma_{\mathrm{in}}}
\newcommand{\Ho}{H^1_{0,out}(\Omega)}

\theoremstyle{plain}% default
\newtheorem{thm}{Theorem}[section]
\newtheorem{lem}[thm]{Lemma}
\newtheorem{prop}[thm]{Proposition}
\newtheorem{cor}[thm]{Corollary}

\theoremstyle{definition}
\newtheorem{defn}{Definition}[section]

\theoremstyle{remark}
\newtheorem{rem}[thm]{Remark}

\title[Uniqueness of solutions to $2D$ Euler with sources and sinks]{Uniqueness of Yudovich's solutions to the 2D incompressible Euler equation despite the presence of  sources and sinks}
 \author{Florent Noisette, Franck Sueur}
  \date{\today}
\address{Florent Noisette, Institut de Math\'ematiques de Bordeaux, UMR CNRS 5251,Universit\'e de Bordeaux, 351 cours de la Lib\'eration, F33405 Talence Cedex, France}
\address{Franck Sueur, Institut de Math\'ematiques de Bordeaux, UMR CNRS 5251, Universit\'e de Bordeaux, 351 cours de la Lib\'eration, F33405 Talence Cedex, France   $\&$ Institut  Universitaire de France}

\begin{document}
\maketitle

\begin{abstract}
In $1962$, Yudovich proved the existence and uniqueness of classical solutions to the 2D incompressible Euler equations in the case where the fluid occupies a bounded domain with entering and exiting flows on some parts of the boundary.
The normal velocity is prescribed on the whole boundary, as well as the entering vorticity.
The uniqueness part of Yudovich's result holds for H\"older vorticity, by contrast with his 1961 result on the case of an impermeable boundary, for which the normal velocity is prescribed as zero on the boundary, and for which the assumption that the initial vorticity is bounded was shown to be sufficient to guarantee uniqueness. 
Whether or not uniqueness holds as well for bounded vorticities in the case of entering and exiting flows has  been left open until $2014$, when Weigant and Papin succeeded to tackle the case where the domain is a rectangle. 
In this paper we adapt Weigant and Papin's result to the case of a smooth domain with several internal sources and sinks.
\end{abstract}

%%%%%%%%%%%%%%%%%%%%%%%%%%%%%%%%%%%%%%%%%%%%%%%%%%%%%%%%%%%%
\section{Introduction}

This first section is devoted to the presentation of the model  and of the mathematical problem which are at stake in this paper. 

\subsection{Geometry of the domain}
We consider a bounded domain $\Omega \subset \R^2$ whose  boundary, denoted by $\Gamma$, is a $C^2$ simple curve which can be decomposed as
\begin{equation} \label{deco-do}
\Gamma = \bigcup_{i\in I}{\Gamma_i},
\end{equation}
where  $I$ is a finite set of index which admits a  partition $I=\{0\}\cup I_{\mathrm{in}}\cup I_{\mathrm{out}}$,  
the sets $\Gamma_i$ are the connected components of $\Gamma$, with the convention that $\Gamma_0$ is the external boundary.
We set
$I^* := I_{\mathrm{in}}\cup I_{\mathrm{out}} = I \setminus \{0\} ,$
and
\begin{equation} \label{tison}
\Gi:=\cup_{i\in I_{\mathrm{in}}}{\Gamma_i} \quad \text{ and } \quad \Go:=\cup_{i\in I_{\mathrm{out}}}{\Gamma_i} .
\end{equation}
The indexes refer to the fact that, below, 
$\Gi$ is the zone with entering flux and $\Go$ is the zone with exiting flux, see  \eqref{signeg}.
% There is moreover no flux across the external boundary.

\subsection{The equations at stake}
We assume that the domain $\Omega$ is occupied by an incompressible perfect flow whose evolution is driven by 
 the incompressible Euler equation. More precisely we consider the following transport equation where the scalar function $\omega(t,x)$ denotes the fluid vorticity:
\begin{subequations}
\label{partie-vort}
\begin{align}
\label{e:trans-fort}
\partial_t \omega + u\cdot \nabla \omega     &= 0     &\text{ on } \Omega,       \\
\label{e:omega-ini-cond}
\omega(0,.) &= \omega_0                               &\text{ on } \Omega,       \\
\label{e:ome-bound-cond}
\omega &= \omein                                      &\text{ on } \Gi.
\end{align}
\end{subequations}
where the vector field $u(t,x)$ is the fluid velocity,  given as the solution of the system:
\begin{subequations}
\label{ellip}
\begin{align}
\label{e:u-incomp}
\div u   &= 0                         &\text{ on } \Omega, \\
\curl u  &= \omega                    &\text{ on } \Omega, \\
\label{e:g} u\cdot n &= g                         &\text{ on } \Gamma, \\
\label{e:def-circ}
\int_{\Gamma_i}{u\cdot \tau} &= \mathcal{C}_i   &\text{  for all } i\in I^*, 
\end{align}
\end{subequations}
where the $(\mathcal{C}_i (t))_{i\in I^*}$ are  the {circulations} of the fluids around each connected component $\Gamma_i$, for $ i\in I^*$,  of the boundary, given as 
the solutions of the following Cauchy problem:
\begin{subequations}
\label{partKelvin}
\begin{align}
\label{e:Kelvin}
\mathcal{C}_i'(t) &= -\int_{\Gamma_i}{\omega g}         &\text{  for all } i\in I^*, \\
\mathcal{C}_i(0) &= \mathcal{C}_{i,0}                                     &\text{  for all } i\in I^*.
\end{align}
\end{subequations}
The quantities    $\omega_0$,  $\omein$, $g$ and  $(\mathcal{C}_{i,0})_{i\in I^*}$, which appear in the right hand sides of the equations above, are given data. 
Regarding $g$,  we assume that, at any time $t$, it has zero average on $\Gamma$, which is the compatibility condition associated with the incompressibility, and we also assume,  as hinted above, that
\begin{align} \label{signeg}
g \leq 0 \text{ on } \Gi, \quad
g \geq 0 \text{ on } \Go, \quad
g = 0    \text{ on } \Gamma_0,
\end{align}
which means that $\Gi$ is the part of the boundary with entering flux, and $\Go$ is the part with exiting flux.
Thus  \eqref{e:ome-bound-cond} is a condition on the 
the vorticity which enters in the  domain $\Omega$. 
Since  $\Omega$ is multiply connected, it is necessary to prescribe the circulations, see  \eqref{e:def-circ}, to guarantee the uniqueness of the solution $u$ to the system  \eqref{ellip}.
These circulations evolve in time according to  Kelvin's law \eqref{e:Kelvin}, where the right hand side encodes  the  vorticity flux accross $\Gamma_i$.
We refer to \cite[Lemma 1.2]{yudo-given} and to \cite[Section 1.3]{BS} for a derivation of \eqref{e:Kelvin} from the velocity formulation of the incompressible Euler equation.
Observe that for $i$ in $I_{in}$, the trace of the vorticity on $\Gamma_i$, 
which appears in the right hand side of \eqref{e:Kelvin} is prescribed according to \eqref{e:ome-bound-cond},
whereas it is part of the  solution for $i$ in $I_{out}$.

\begin{figure} \centering  
\includegraphics[scale= 0.4]{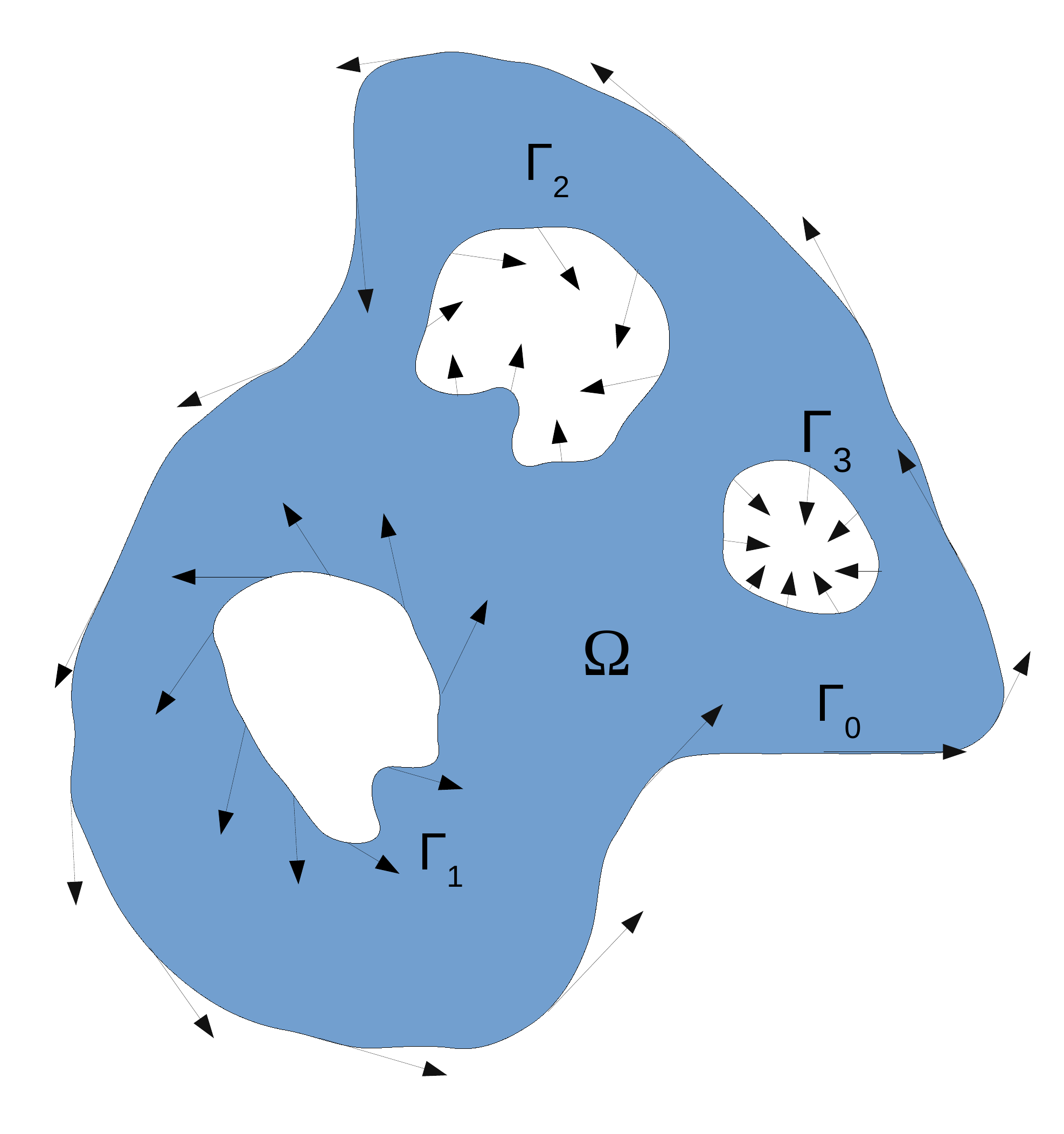}
 \caption{Example of a fluid domain with one source and two sinks.} \label{fig} 
  \end{figure}

\subsection{An open problem on the uniqueness of solutions with bounded vorticity}

The system \eqref{partie-vort}-\eqref{ellip}-\eqref{partKelvin} was proposed by  Yudovich  in \cite{yudo-given}, who proved the existence and uniqueness of classical solutions. 
Later, the existence part of Yudovich's result in the permeable case has been generalized  to weaker solutions, 
see  \cite{Alekseev,BS,Che:Sta,Uvarovskaya}. 
On the other hand, for a long time, no progress has been obtained on the uniqueness part despite that the conjecture that uniqueness should hold in the case where the vorticity is bounded, was broadly shared. Indeed such a result is well-known  in the case of impermeable boundaries, for which $u\cdot n =0$ on $\Gamma$ instead of \eqref{e:g}, as proved by Yudovich in his celebrated result  \cite{yudo63}. 
This problem was recalled for instance in \cite[Section 3.2]{Mamontov} and in  \cite[Section 3.2.1]{Glass}. Indeed this open problem has known a regain of interest in the controllability community after the works on the incompressible Euler system by  Coron and Glass in \cite{coron1,coron2,coron3,glass2001}.
Finally, a breakthrough result was obtained by  Weigant and Papin in $2014$, see \cite{pap-weig:ini}, who proved the case where the fluid domain is a rectangle with lateral inlet and outlet, under the assumption that the vorticity is bounded.  
Their nice proof makes use of two energy estimates on the difference of two solutions, one related to the time-evolution of the kinetic energy, and another one associated with a clever auxiliary function. The two estimates are combined to deduce a stability estimate which in particular guarantees uniqueness of the solutions with bounded vorticities for this particular geometry.  
In this paper we extend Weigant and Papin's  approach to prove the uniqueness of solutions with bounded vorticity in the 
case, presented above, of a smooth multiply-connected domain with several interior sources and sinks.

%%%%%%%%%%%%%%%%%%%%%%%%%%%%%%%%%%%%%%%%%%%%%%%%%%%%%%%%%%%%
\section{Reminder on the definition and existence of solutions with bounded vorticity}

This section is devoted to recall a definition and an existence result of weak solutions to the system  \eqref{partie-vort}-\eqref{ellip}-\eqref{partKelvin}
 in the case where the vorticity is bounded.
This will set up the context to which the uniqueness result of this paper is applied. 

Let us start the discussion with the choice of the unknowns. 
Because of the boundary condition \eqref{e:ome-bound-cond}, the vorticity $\omega$ is a natural choice, for instance compared to the velocity $u$. 
Also, as mentioned above, the vorticity which flows through  $\Go$ is unknown; 
we  denote it by $\omeout$. 
Taking \eqref{e:ome-bound-cond} into account for what concerns $\Gi$, 
Kelvin's laws \eqref{partKelvin} can therefore be recast as  
\begin{subequations}
\begin{align}
\label{circu-in}
\mathcal{C}_i(t) &= \mathcal{C}_{i,0} -\int_{0}^{t}{\int_{\Gamma_i}{\omein \,  g}}  &\text{ for } i\in I_{\mathrm{in}}, \\
\label{circu-out}
\mathcal{C}_i(t) &= {\mathcal{C}}_{i,0} -\int_{0}^{t}{\int_{\Gamma_i}{\omeout \,  g}}  &\text{ for } i\in I_{\mathrm{out}}.
\end{align}
\end{subequations}
The fact that the system \eqref{ellip} determines a single solution $u$, the right hand sides being given, is well known. 
In particular, we have the following result, see for instance \cite{yudo-given},
where we omit the  time-dependence since the time only plays the role of a parameter in this part of the system. 
\begin{prop}
\label{p:yudo-ell-bound}
Let $(\mathcal{C}_i)_i\in \R^{I^*}$, $g$ in the H\"older space $C^{1,\alpha}(\Gamma)$ for some $\alpha\in (0,1)$, with zero average on $\Gamma$, and $\omega$  in $L^{\infty}(\Omega)$.
 Then the system \eqref{ellip}  admits an unique solution 
 $$u \in \underset{1\leq p<+\infty}{\cap}   W^{1,p}(\Omega) .$$
Moreover there is a constant $C>0$ such that, for all $p\geq 2$, 
\begin{equation}
\label{pbound}
\Vert u \Vert_{W^{1,p}(\Omega)}
    \leq Cp\big(\Vert\omega\Vert_{L^p(\Omega)}+\Vert g\Vert_{C^{1,\alpha}(\Gamma)}+\sum_{i\in I^*}{|\mathcal{C}_i|}\big).
\end{equation}
\end{prop}
\begin{rem} \label{rk-c-p}
The dependence with respect to $p$ in \eqref{pbound} is useful in Section \ref{sec-game}. Indeed  that the constant
 in front of the parenthesis in the right hand side  is bounded by $p$ for large $p$ 
 is crucial in Yudovich's proof of uniqueness in the impermeable case, as it allows to bypass the failure of the Lipschitz estimate of $u$, which is the main difficulty with respect to case of  classical solutions.  To overcome this difficulty, Yudovich used some \textit{a priori} bounds on  the $L^p$ norms of the vorticity.   We refer here to \cite{yudo1995,yudo2005} for more. 
\end{rem}

Finally it only remains to tackle the transport part \eqref{partie-vort} of the 
 system \eqref{partie-vort}-\eqref{ellip}-\eqref{partKelvin}. When considering this initial boundary value problem on the time interval $[0,T]$, with $T>0$, we are naturally led to a weak formulation corresponding to the following family of identities:  
\begin{equation}\label{e:wf-trans}
\int_{t_0}^{t_1}{\int_{\Omega}{\omega\left(\partial_t\phi+u\cdot\nabla\phi\right)}}
    = \left\lbrack \int_{\Omega}{\omega\phi} \right\rbrack_{t_0}^{t_1}
     + \int_{t_0}^{t_1}{\int_{\Gi}{\omein\phi  g}}
     + \int_{t_0}^{t_1}{\int_{\Go}{\omeout\phi  g}},
\end{equation}
for  $0\leq t_0<t_1 \leq T$,  and 
 for some  test functions $\phi$ to determine, and where the notation $[a]_{t_0}^{t_1}$ means 
\begin{equation}
\label{jump}
[a]_{t_0}^{t_1} := a(t_1)-a(t_0) .
\end{equation}
\begin{defn}\label{d:def-weaksol}
Let $T>0$. Let  $\omega_0\in L^{\infty}(\Omega)$, $(\mathcal{C}_{i,0})_{i\in I^*}\in \R^{I^*}$,
$g$ in $L^{\infty}([0,T] ;C^{1,\alpha}(\Gamma))$ 
with zero average on $\Gamma$ at every time,
and $\omein$ in $  L^{\infty}([0,T] ; L^{\infty}(\Gi, g))$.
We say that 
\begin{equation}
\label{xprime}
(\omega,\omeout)\in C^{0}([0,T] ; L^{\infty}(\Omega) - w_*) \times L^{\infty}([0,T] ;L^{\infty}(\Go, |g|)),
\end{equation}
is a weak solution to the  system \eqref{partie-vort}-\eqref{ellip}-\eqref{partKelvin}
when  for every test function $\phi\in H^1([0,T] \times\Omega;\R)$, and for  every $t_0<t_1$, 
the equality \eqref{e:wf-trans} holds true with $u$  given, at time $t$ in $[0,T]$, as the unique solution of the system \eqref{ellip} given by Proposition  \ref{p:yudo-ell-bound}, 
where the circulations are given by \eqref{circu-in}-\eqref{circu-out}.
\end{defn}
Let us emphasize that the appearance of $g$ in the notation $L^{p}(\Gamma,  |g|)$ refers to the measure $|g|  \mathcal{H}^1$, where $\mathcal{H}^1$ is the one-dimensional Haussdorf measure on $\Gamma$. On the other hand the notation
$L^{\infty}(\Omega) - w_* $ in \eqref{xprime} 
stands for the space $L^{\infty}(\Omega)$ equipped with
the weak star topology. 
\begin{rem} \label{rem-trace}
The article \cite{Boyer} by Boyer allows to give a sense to the trace of $\omega$ on the permeable part of the boundary, that is where $g\neq 0$, for solutions, in the sense of distributions, of the transport equation \eqref{e:trans-fort}.  
Therefore, when it does not lead to confusion, we use the notation $\omega$ instead of  $\omein$ or $\omeout$  on the corresponding parts $\Gi$ and $\Go$ of the boundary. 

It also establishes that the assumption regarding the  continuity in time in \eqref{xprime}  is not restrictive. 
\end{rem}

The existence of weak solutions to the  system \eqref{partie-vort}-\eqref{ellip}-\eqref{partKelvin}  in the sense of Definition \ref{d:def-weaksol} has been obtained in 
\cite[Theorem 1]{Uvarovskaya}, see also \cite[Theorem 3]{BS} for a slightly different proof, and for some other existence results of weaker solutions.

%%%%%%%%%%%%%%%%%%%%%%%%%%%%%%%%%%%%%%%%%%%%%%%%%%%%%%%%%%%%
\section{Statement of the main result}
\label{sec-st}

This section is devoted to  the statement of the main result of the paper. We also make a few remarks about it and finally explain how is organized of the rest of the paper.

\subsection{Statement of the main result and a few remarks}
The following result establishes a quantitative stability estimate for weak solutions  to the  system \eqref{partie-vort}-\eqref{ellip}-\eqref{partKelvin}
  in the sense of Definition \ref{d:def-weaksol},  corresponding to  different initial and boundary data, which in particular implies uniqueness for solutions  corresponding to the same initial and boundary data. 
\begin{thm} \label{main-thm}
Let $T>0$ and two weak solutions to the  system \eqref{partie-vort}-\eqref{ellip}-\eqref{partKelvin} as in Definition \ref{d:def-weaksol} corresponding to $T$, to the same boundary data $g$ for the normal velocity, and to some possibly different data.
We denote those initial and boundary data  respectively 
  $\omein^1$ and $\omein^2$ for the entering vorticity;  
  $(\mathcal{C}^1_{i,0})_{i\in I^*}$ and $(\mathcal{C}^2_{i,0})_{i\in I^*}$ for the initial circulations;
  and $\omega^1_0$ and $\omega^2_0$ for the initial vorticity. 
Let $u^1$ and $u^2$ the corresponding velocities as given by Proposition  \ref{p:yudo-ell-bound}. 
Then  there exists a continuous function
$F: (\R_+)^3  \times \R^{I^*} \mapsto \R_+$ satisfying 
$F(t,0,0,0)=0$ for all $t \geq 0$, 
such that  for all $t$ in  $[0,T]$,
\begin{gather} \label{stab-est}
     \Vert u^1(t,.) - u^2(t,.) \Vert_{L^2 (\Omega)} 
      +  \int_{0}^{t} \int_{\Gamma} | u^1 - u^2  |^2| g|
      \\ \nonumber \quad \leq 
F\big(t,\Vert\omega^1_0-\omega^2_0\Vert_{L^{\infty}(\Omega)} , \Vert\omein^1-\omein^2\Vert_{L^{\infty}([0,t];L^{\infty}(\Gi,|g|))},
(\mathcal{C}^1_{i,0} - \mathcal{C}^2_{i,0})_{i\in I^*} \big).
\end{gather}
In particular,  for some given initial and boundary data, there exists an unique weak solution   to the  system \eqref{partie-vort}-\eqref{ellip}-\eqref{partKelvin} in the sense of Definition \ref{d:def-weaksol}. 
\end{thm}

\begin{rem}
The results of Theorem \ref{main-thm} can be extended to the slightly more general case where the vorticity is unbounded but with a moderate growth of its $L^p$ norms, as it was done by Yudovich in \cite{yudo2005,yudo1995}. Indeed the limitation in terms of regularity assumption on the vorticity is due to interior terms, which are treated similarly in the permeable and in the impermeable case. As a matter of fact, to extend Theorem \ref{main-thm} to the same setting as in \cite{yudo2005,yudo1995}, it is sufficient to adapt the Osgood argument used  in section \ref{sec-game}. 
\end{rem}

\begin{rem}
Theorem \ref{main-thm} can be extended to the case where the boundary datum $g$ for the normal velocity on the boundaries of the sources and of the sinks oscillates in time, as long as the sign is the same, for  every time, on each connected component. It is therefore possible to consider some cases where a inner connected component of the  boundary is at some time a source and at another time a sink.  On the other hand, transitions  between an inflow part and an outflow part in space, on the same connected component of a boundary, seem to require careful studies. Let us highlight  that such transitions occur in \cite{pap-weig:ini},  with right angles at the places where transitions occur.
\end{rem}

\begin{rem}
Let us insist on the fact that the stability estimate \eqref{stab-est} concerns two solutions to the  system \eqref{partie-vort}-\eqref{ellip}-\eqref{partKelvin}
corresponding to the same boundary data $g$ for the normal velocity. 
The analysis performed below does not consider the issue of the stability of the solutions to the  system \eqref{partie-vort}-\eqref{ellip}-\eqref{partKelvin} with respect to perturbations of $g$. This issue is of interest with respect to the  stabilization  issue \cite{coron3,glass2001}.
\end{rem}

\begin{rem}
In the impermeable case, several proofs of uniqueness are available, with quantitative estimates in different topologies for different quantities. 
More precisely, while the original proof by Yudovich in \cite{yudo63} relies on an energy estimate, that is on the $L^2$ norm of the fluid velocity,
some alternative proofs have been since found, in particular thanks to a Lagrangian viewpoint in  \cite[Theorem 3.1]{Marc-Pulv} by Marchioro and Pulvirenti, where the $L^1$ norm of the flow map is used, and thanks to tools of optimal transportation theory in \cite{Loeper} by Loeper with the  Wasserstein distance $W_2$ and in \cite{Hauray} by Hauray with the  Wasserstein distance $W_\infty$.
While the proof of Theorem \ref{main-thm} given below uses an energy-type argument,  it would be interesting to investigate whether an alternative proof of Theorem \ref{main-thm}  based on the Lagrangian viewpoint could also be carried on. 
\end{rem}

\begin{rem}
\label{rem-decomp}
Observe that  Theorem \ref{main-thm} implies in particular an energy estimate for a (single) weak solution  to the  system \eqref{partie-vort}-\eqref{ellip}-\eqref{partKelvin}
  in the sense of Definition \ref{d:def-weaksol},  by considering the solution  to the  system \eqref{partie-vort}-\eqref{ellip}-\eqref{partKelvin} corresponding to the same initial circulations and to the same prescribed trace $g$ for the normal velocity but with zero initial and entering vorticity, and therefore with zero vorticity in $\Omega$ and on $\Go$ at any time.
\end{rem} 
 
\begin{rem}
 Theorem \ref{main-thm} implies in particular the uniqueness of the unstationary Euler system with prescribed normal velocity and entering vorticity. 
Let us recall that, on the other hand, it is known that  solutions of the 2D stationary Euler system with prescribed normal velocity and entering vorticity are not in general unique (see e.g. \cite{Troshkin}).
\end{rem}

\begin{rem}
A natural question is whether the result in  Theorem \ref{main-thm} and the method used in its proof  can be extended to some other nonlinear evolution PDEs with non-conservative boundary conditions, in particular to the ones which share the same features to couple transport and non-local features. A candidate in this direction is the Camassa-Holm equation, set on a finite interval with some 
inhomogeneous boundary conditions  as considered in \cite{Perrollaz}. An open question left in this paper was the uniqueness of the weak solutions obtained in \cite[Theorem 1]{Perrollaz}. 
\end{rem}

\begin{rem}
Finally let us mention that the weak-strong uniqueness property is a natural issue which has not been investigated yet for the  system \eqref{partie-vort}-\eqref{ellip}-\eqref{partKelvin}. In the case of impermeable boundaries, it is associated with the notion of dissipative solutions, see 
\cite[chapter 4.4]{Lions}. Recently the effect of an impermeable boundary  on the issue of the weak-strong uniqueness 
has been investigated by various authors, we refer here to the survey \cite{Wiedemann}
for more on the subject  and it would be therefore interesting to extend these investigations to the case of permeable boundaries. 

\end{rem}

\subsection{\textbf{Strategy of the proof of Theorem \ref{main-thm} and organization of the rest of the paper.}}

The rest of the paper is devoted to the proof of Theorem \ref{main-thm} which compares the difference of two solutions of the Euler equations in presence of sources and sinks.
The set-up of the proof of Theorem \ref{main-thm}, with the derivations of the equations satisfied by the difference of two solutions is done in Section \ref{s:1}. 
The initial idea is to perform an energy estimate from the weak vorticity formulation by using a stream function of the difference as a test function.
In the case where the fluid occupies the whole plane, with nice decay at infinity, this corresponds to the identity: 
\begin{equation*} \label{ pseudoe}
\int_{\mathbb R^2} u \cdot u = - \int_{\R^2} \psi \omega ,  \quad  \text{ where } \quad \omega = \curl u = \Delta \psi  .
\end{equation*}
This allows to bypass the velocity formulation which has some unpleasant features in the permeable case, in particular due to the pressure.  
A technical difficulty in this process is to justify that the stream function associated with the difference of the two solutions at stake is regular enough to be taken as a test function in the weak formulation of the equation. 
The estimate of the stream function of the difference is performed in Section \ref{sec-tPsi}.
The most delicate part is to obtain some estimates of the time-derivative of the stream function. 
This is accomplished thanks to the weak vorticity formulation with some other appropriate particular test functions associated with the geometry. 
Then the energy estimate of the difference is performed in Section \ref{p:ener-eq}.
A difficulty is the presence of a "bad" boundary term corresponding to "the energy entering at the sources". 
Because of its sign, it is needed to bound this term for the energy estimate  to be conclusive. 

A great idea, first used by Weigant and Papin in the case where the fluid domain is a rectangle with lateral inlet and outlet, see \cite{pap-weig:ini},
is to couple this energy estimate with a second one, obtained with the help of an appropriate test function. 
As in \cite{pap-weig:ini},  we consider a harmonic test function with mixed boundary conditions, more precisely with Neumann inhomogeneous conditions on the boundary of the sources and with zero Dirichlet condition on the rest of the boundary.
However, compared  to the case of \cite{pap-weig:ini} where this test function is constructed and estimated thanks to some  Fourier series,
the construction and the estimates in the present case requires more works, which are done in Section \ref{s:2}. 
In  Section \ref{Lamb-G}, we state a generalized Lamb's lemma which tackles trilinear integrals  by some appropriate vector calculus identities and some integrations by parts.
This allows to clarify the treatment of some convective terms in the second energy estimate associated with the auxiliary harmonic test function.
The auxiliary  energy-type estimate is performed in Section \ref{sec-aux-esti}. 
Finally, we combine the two energy estimates in Section \ref{sec-game} and concludes the proof of Theorem \ref{main-thm} by Osgood's lemma.

%%%%%%%%%%%%%%%%%%%%%%%%%%%%%%%%%%%%%%%%%%%%%%%%
\section{Equations satisfied by the difference of two solutions}
\label{s:1}

To prove Theorem \ref{main-thm}, we use an energy method to compare the dynamics of two weak solutions of the Euler equation in the sense of  Definition \ref{d:def-weaksol} corresponding to $T>0$. 
To that extent, let us fix two solutions $(\omega^1,\omeout^1)$ and $(\omega^2,\omeout^2)$ of the Euler equation, with initial and boundary conditions %
$$(\omega_0^1,(\mathcal{C}_{i,0}^1)_{i\in I^*},g,\omein^1)\quad \text{ and } \quad
(\omega_0^2,(\mathcal{C}_{i,0}^2)_{i\in I^*},g,\omein^2).$$
We consider $u^1$ and $u^2$ the  velocities which are respectively associated to the previous quantities by Proposition  \ref{p:yudo-ell-bound}.
We use the following notations:
\begin{align*}
\tomega := \omega^1-\omega^2  \quad \text{ and } \quad
\homega  := \frac{\omega^1+\omega^2}{2},
\\ \tu   := u^1-u^2 \quad \text{ and } \quad
\hu  := \frac{u^1+u^2}{2},
\end{align*}
By linearity of \eqref{e:u-incomp}
and \eqref{partKelvin} the vector field $\tu   $ satisfies, at every time,  the 
system:
\begin{subequations}
\label{ellip-diff}
\begin{align}
\label{dusoleil}
\div \tu     &= 0                         &\text{ in } \Omega, \\
\curl \tu    &=    \tomega                &\text{ in } \Omega, \\
\label{imp-diff} \tu   \cdot n &= 0                      &\text{ on } \Gamma, \\
\label{circu-u-diff} 
\int_{\Gamma_i}{ \tu   \cdot \tau} &= \tilde{\mathcal{C}}_i   &\text{  for all } i\in I^*, 
\end{align}
\end{subequations}
where, for $i$ in $I^*$,  
the circulation $\tilde{\mathcal{C}}_i$ of $\tu  $ around $\Gamma_i$ which satisfies at any $t\in [0,T]$, 
\begin{equation}
\label{noel}
\tilde{\mathcal{C}}_i(t) = \tilde{\mathcal{C}}_{i,0} -\int_{0}^{t}{\int_{\Gamma_i}{\tomega g}}, 
\text{ with } \tilde{\mathcal{C}}_{i,0} :=
\mathcal{C}_{i,0}^1 - \mathcal{C}_{i,0}^2 ,\quad
\text{ for } i\in I^*.
\end{equation}
Observe that, as hinted in Remark  \ref{rem-trace}, we use the notation $\tomega$ in the first term of the right hand side of \eqref{noel} instead of  $\tomein :=\omein^1 - \omein^2$ or $\tomeout :=  \omeout^1 -  \omeout^2$  on the corresponding parts $\Gi$ and $\Go$ of the boundary $\Gamma$.

Using the Hodge-De Rham theory, we decompose, at every time,  the vector field $\tu  $ into two types of contributions respectively corresponding to the circulations and to the vorticity.
\begin{itemize}
\item 
On the one hand, to tackle the effect of the circulations,  we consider, for $i$ in $I^*$, the functions $f^i$, which are the unique solutions in $C^2  (\overline \Omega ) $ of the following boundary value problem:
\begin{align}\label{def-f}
\Delta f^i = 0 \text{ on } \Omega,  \quad \text{and} \quad
f^i = \delta_{i,j} \text{ on } \Gamma_j,
\end{align}
where $\delta_{i,j}$ is  equal to $1$ when $i=j$ and to $0$ otherwise.
\item  
On the other hand, we consider, for any given bounded function $\omega$, 
the unique solution $G[\,\omega]$ of the following boundary value problem:
\begin{align}\label{def-Gd}
\Delta G[\,\omega] = \omega \text{ on } \Omega,  \quad \text{and} \quad 
G[\,\omega] = 0 \text{ on } \Gamma ,
\end{align}
and we set 
$$K[\,\omega]:= \nabla^{\perp} G[\,\omega].$$  
\end{itemize}
It is a classical result, see for instance \cite{Kato}, of the Hodge-De Rham theory that there exists some real-valued functions $\tilde{\psi}_i = \tilde{\psi}_i (t)$, for $i\in I^*$, such that, at every time,  
\begin{equation}
\label{e:u-nab-perp-psi}
\tu (t,\cdot)  = \nabla^{\perp} \tilde{\psi} (t,\cdot) \quad  \text{ in } \,  \Omega,
\end{equation}
where
\begin{equation}\label{e:psi-decomp}
\tilde{\psi} (t,\cdot):= G[\, \tomega(t,\cdot)] + \sum_{i\in I^*}{\tilde{\psi}_i (t) f^i} .
\end{equation}
Let us observe that, as a consequence of \eqref{dusoleil} and  \eqref{e:u-nab-perp-psi}, %
\begin{equation}
\label{delta-hat} \Delta \tilde{\psi} = \tomega     \quad \text{ in } \Omega, 
\end{equation}
and, as a consequence of \eqref{def-f} and \eqref{def-Gd}, 
\begin{equation}
\tilde{\psi}_{|\Gamma_i} = \tilde{\psi}_i  , \quad \text{  for all } \,  i\in I \label{aubord} ,
\end{equation}
where we set  $\tilde{\psi}_0 := 0$. 
 Combining \eqref{circu-u-diff} and  \eqref{e:u-nab-perp-psi}, 
we arrive at 
\begin{equation}
\label{circu-psi-diff}
\int_{\Gamma_i}{ \partial_n \tilde{\psi} } = \tilde{\mathcal{C}}_i   \quad \text{  for all } \, i\in I^* .
\end{equation}
 Moreover, for all test function $\phi\in H^1\left([0,T]\times \Omega\right)$,  and for  every $0 \leq t_0<t_1 \leq T$, 
\begin{equation}\label{e:weak-trans-diff}
\int_{t_0}^{t_1}{\int_{\Omega}{\left(\tomega\partial_t \phi + \tomega \hu \cdot \nabla \phi  +\homega \tu  \cdot  \nabla \phi\right)}}
    = \int_{t_0}^{t_1}{\int_{\Gamma}{\phi \tomega g}}
    + \left[\int_{\Omega}{\tomega \phi} \right]_{t_0}^{t_1} .
\end{equation}
On the other hand, the half-sum $\hu $ satisfies 
\begin{equation}
\label{imp-moy}
\hu \cdot n    = g                \quad \text{ on } \Gamma .
%, \\ \label{sertoupa} \homega  &=  \tilde{\omein}   &\text{ on } \Gamma .
\end{equation}
%

%%%%%%%%%%%%%%%%%%%%%%%%%%%%%%%%%%%%%%%%%%%%%%%%%%%%%%
\section{Estimate of the stream function of the difference}
\label{sec-tPsi}

This section is devoted to the regularity of the stream function $\tilde{\psi}$. 
The main result of this section reads as follows. 
\begin{prop} \label{prop-f}
The stream function $\tilde{\psi}$ 
is in the space $C^0 ( [0,T] ;W^{2,p}(\Omega))$ for all $p$ in $[1,+\infty)$ and 
 in the space $H^1 ([0,T] ; H^1 (\Omega ))$.
Moreover, there exists $C>0$ such that 
for $i\in I$, the time derivative of the trace $\tilde{\psi}_i$ of $\tilde{\psi}$ on~$\Gamma_i$ satisfies 
 for  $0 \leq t_0<t_1 \leq T$,
\begin{equation} \label{desire}
\int_{t_0}^{t_1}{|\tilde{\psi}_i'|^2}
    \leq C
       \left(
       \int_{t_0}^{t_1}{\Vert  \tu  \Vert_{L^{2} (\Omega)}^2}
       +\int_{t_0}^{t_1}{\Vert  \tu  \Vert_{L^2(\Gamma, |g|)}^2}
       \right).
\end{equation}
\end{prop}
Proposition \ref{prop-f} is useful in  Section \ref{p:ener-eq} to obtain an energy estimate, see Proposition \ref{t:energ-eq}, 
as it allows us to apply \eqref{e:weak-trans-diff} to the case where the test function is the stream function $\tilde{\psi}$.
On the other hand the estimate \eqref{desire} is useful in Section    \ref{sec-aux-esti}. Let us highlight that it follows from the regularity of the solutions at stake that the right hand side of \eqref{desire} is finite.

\begin{proof}
The  fact that  the function $\tilde{\psi}$ is in the space  $C^0 ([0,T];W^{2,p}(\Omega))$ for all
 for all $p$ in $[1,+\infty)$ can be found in \cite[Lemma 1.4]{yudo-given}.
  
We now turn to the estimate of the time derivative of  $(\tilde{\psi}_i)_i$, for  $i\in I$.
For $i=0$, we have by definition that $\tilde{\psi}_0=0$, so that
 \eqref{desire} holds true in this case with any constant $C>0$ and 
it is therefore sufficient to deal with the case where  $i\in I^*$.
To this aim, we recall that there exist some  constants $(g^i_j)_{i\in I^* , \, j\in I}$, with $g_0^i = 0$,  and some smooth functions $(g^i)_{i\in I^*}$ such that
\begin{subequations} \label{def-g}
\begin{align}
\Delta g^i &= 0 \quad \text{ on } \Omega, \\
g^i &= g^i_j \quad   \text{ on } \Gamma_j ,\quad  \text{ for } j\in I , \label{def-g2} \\
\int_{\Gamma_j}{ \pn g^i } &= - \delta_{i,j} \quad  \text{ for } j\in I^* .
\end{align}
\end{subequations}
These functions can be obtained by some appropriate linear combinations of the functions $f^i$, which are defined in \eqref{def-f}, we refer here again to \cite{Kato} for more details.

\begin{lem} \label{lemmaprime}
For any $i$ in $I^*$, the time derivative  $\tilde{\psi}'_i$ is in the space $L^2(0,T)$ and 
there exists $C>0$ such that for any $0\leq t_0<t_1 \leq T$, the estimate \eqref{desire} holds true.
\end{lem}

\begin{proof}[Proof of Lemma \ref{lemmaprime}]
Let  $i$ in $I^*$.
We take $g^i$ as a (time-independent) test function in \eqref{e:weak-trans-diff}, and obtain that for every $0\leq t_0<t_1 \leq T$, 
\begin{equation} \label{test-g}
\int_{t_0}^{t_1}{\int_{\Omega}{\left( \tomega \hu \cdot \nabla g^i  +\homega  \tu  \cdot  \nabla g^i\right)}}
    = \int_{t_0}^{t_1}{\int_{\Gamma}{g^i \tomega g}}
    +
   \left[ \int_{\Omega}{\tomega g^i}\right]_{t_0}^{t_1} .
\end{equation}

Let us start with the right hand side of \eqref{test-g}. 
By \eqref{deco-do} and \eqref{def-g2}, the first term in the right hand side of \eqref{test-g} can be simplified as follows:
\begin{align*}
\int_{t_0}^{t_1}{\int_{\Gamma}{g^i \tomega g}}
    = \sum_{j\in I^*} g^i_j   \int_{t_0}^{t_1}{\int_{\Gamma_j}{\tomega g}}
  = - \sum_{j\in I^*} g^i_j   \left[   \tilde{\mathcal{C}}_j\right]_{t_0}^{t_1},
\end{align*}
thanks to \eqref{noel}.

On the other hand, using \eqref{delta-hat}  and two integrations by parts,
the second  term in the right hand side of \eqref{test-g} can be simplified by observing that:
\begin{align*}
\int_{\Omega}{\tomega(t,.)g^i} 
    &= \int_{\Omega}{\tilde{\psi}(t,.)\Delta g^i}
     + \int_{\Gamma}{\pn \tilde{\psi}(t,.) g^i} 
     - \int_{\Gamma}{\tilde{\psi}(t,.)\pn g^i} \\
    &= \sum_{j\in I^*}  g^i_j \tilde{\mathcal{C}}_j(t) + \tilde{\psi}_i(t) ,
\end{align*}
thanks to \eqref{aubord},  \eqref{circu-psi-diff} and \eqref{def-g}.
Thus the right hand side of \eqref{test-g} is equal to $\tilde{\psi}_i(t_1) - \tilde{\psi}_i(t_0)$.
  
Regarding  the second term in the left hand side  of \eqref{test-g}, by the Cauchy-Schwarz inequality, we obtain that
\begin{align} \label{enplus}
\left\vert  {\int_{\Omega}{\homega  \tu  \cdot  \nabla g^i}} \right\vert 
    &\leq  \Vert  \nabla g^i \Vert_{L^{\infty} (\Omega)}  \Vert\homega \Vert_{L^2 (\Omega)}   \Vert  \tu   \Vert_{ L^{2} (\Omega)} .
\end{align} 
Since  $\Vert\homega \Vert_{L^{\infty}( [0,T] ; L^2 (\Omega))} $ is bounded,
we infer that the left hand side of \eqref{enplus} is in $L^2 (0,T)$ 
and that there  exists $C>0$ such that  for  $0~\leq~t_0<t_1 \leq~T$
\begin{align}  \label{holder1} 
 \int_{t_0}^{t_1} \left\vert {\int_{\Omega}{\homega  \tu  \cdot  \nabla g^i}}\right\vert^2
    &\leq  C \int_{t_0}^{t_1} {\Vert  \tu  \Vert_{L^{2} (\Omega)}^2} . 
\end{align} 
On the other hand, regarding the first term in the left-hand-side  of \eqref{test-g},
we integrate by parts by observing that $\tomega = \div \tu  ^{\perp}$, so that 
\begin{equation}
\label{dim}
{\int_{\Omega}{\tomega \hu \cdot \nabla g^i }}
    = 
  -  {  \int_{\Omega}{ \tu  ^{\perp}\cdot \nabla\left(\hu \cdot \nabla g^i\right)}}
    +{\int_{\Gamma}{( \tu  \cdot \tau)\left(\hu \cdot \nabla g^i\right)}}.
\end{equation}

Regarding the first term in the right-hand-side of \eqref{dim}, by Leibniz' rule and the Cauchy-Schwarz inequality, we arrive at:
\begin{equation*}
\left\vert{  \int_{\Omega}{ \tu  ^{\perp} \cdot  \nabla\left(\hu \cdot \nabla g^i\right)}}\right\vert
  \leq\Vert  \tu   \Vert_{ L^2 (\Omega)} \left(\Vert \nabla g^i \Vert_{L^{\infty} (\Omega)}  \Vert\hu \Vert_{ H^{1} (\Omega)}  
 +\Vert \nabla^2 g^i \Vert_{L^{2}_x}  \Vert\hu \Vert_{L^{\infty} (\Omega)}
 \right) .
 \end{equation*}
Since the term in the parenthesis in the right-hand-side is in $L^{\infty} ([0,T] )$,
the right-hand-side is in the space $L^2(0,T)$ and there exists $C>0$ such that  for  $0\leq t_0<t_1 \leq T$,
\begin{equation} \label{holder2}
\int_{t_0}^{t_1} \left\vert{  \int_{\Omega}{ \tu  ^{\perp} \cdot  \nabla\left(\hu \cdot \nabla g^i\right)}}\right\vert^2 
   \leq  C\int_{t_0}^{t_1} {\Vert  \tu  \Vert_{L^{2} (\Omega)}^2} . 
\end{equation}

On the other hand, for the second term in the left hand side of \eqref{dim}, 
we proceed as follows. Let us set 
$$\Upsilon := \int_{\Gamma}{( \tu  \cdot \tau)\left(\hu \cdot \nabla g^i\right)},$$
which is a time-dependent function. 
Since the function $g^i$ is constant on each connected component of the boundary  $\Gamma$, 
$$\hu \cdot \nabla g^i = (\hu \cdot n)  \partial_n g^i = g    \partial_n g^i ,$$
by \eqref{imp-moy}. 
Then we recast $\Upsilon$ into 
$$\Upsilon ={\int_{\Gamma}}{\Big(( \tu  \cdot \tau) \vert g \vert^\frac12 \Big) \Big( ( \pn g^i ) \textrm{Sign}(g) \vert g \vert^\frac12 \Big)}
,$$
so that,  by the Cauchy-Schwarz inequality, 
\begin{align*}
\vert \Upsilon \vert \leqslant
    \Big( \int_{\Gamma}(\tu \cdot \tau)^2 \vert g \vert\Big)^\frac12 
    \Big( \int_{\Gamma}(\pn g^i )^2 \vert g \vert \Big)^{\frac{1}{2}} .
\end{align*}
 Therefore 
\begin{align}
\label{holder3} 
\int_{t_0}^{t_1}{|\Upsilon|^2}  
    \leq \int_{t_0}^{t_1}{\Vert\tu\Vert_{L^2(\Gamma,|g|))}^2}   . \int_{\Gamma} (\pn g^i)^2 \vert g \vert 
\leq C \int_{t_0}^{t_1}{\Vert\tu\Vert_{L^2(\Gamma,|g|))}^2}.
\end{align}
Gathering \eqref{test-g}, \eqref{holder1}, \eqref{holder2} and \eqref{holder3} for  $0 \leq t_0<t_1 \leq T$, 
we obtain that $\tilde{\psi}'_i $ is in $L^{2}(0,T)$ and the estimate \eqref{desire} hold true.
\end{proof}

We now prove  the following result on the function $G[\, \tomega]$,  recalling that the operator $G[\, \cdot]$ is defined in \eqref{def-Gd}. For the purpose of proving Proposition \ref{prop-f}, it would be enough to prove that $G[\, \tomega]$ is in the space  $H^1 ([0,T] ;H^1(\Omega))$; however it does not involve more work to obtain a slightly better regularity in time. 
\begin{lem} \label{lemB}
The function $G[\, \tomega]$ is in the space $\textrm{Lip}([0,T];H^1(\Omega))$.
\end{lem}

\begin{proof}[Proof of Lemma \ref{lemB}]
Let $a$ in $H^1_0(\Omega)$.
By using  $a$ as a (constant-in-time) test function 
in \eqref{e:weak-trans-diff}, we get that  for  $0 \leq t_0<t_1 \leq T$, 
\begin{align*}
\int_{t_0}^{t_1}{\int_{\Omega}{\left(\tomega \hu  +\homega  \tu  \right)\cdot  \nabla a}}
    &= \int_{\Omega}{(\tomega(t_1,.)-\tomega(t_0,.))a} \\
    &= \int_{\Omega}{\left(\Delta G[\, \tomega](t_1,.)-\Delta G[\, \tomega](t_0,.)\right) a} \\
    &=- \int_{\Omega}{\nabla\left( G[\, \tomega](t_1,.)- G[\, \tomega](t_0,.)\right)\cdot \nabla a},
\end{align*}
by using the definition of the operator $G[\, \cdot]$ in \eqref{def-Gd} and an integration by parts. 
In particular, in the case where 
 $$a=G[\, \tomega](t_1,.)-G[\, \tomega](t_0,.),$$
 we infer by the Cauchy-Schwarz inequality as well as Poincar\'e inequality, see for example \cite[Theorem 13.6.9]{TW},  that there exists a constant $C>0$ such that for  $0 \leq t_0<t_1 \leq T$,  
\begin{equation}
\Vert G[\, \tomega](t_1,.)-G[\, \tomega](t_0,.) \Vert_{H^1 (\Omega)}\leq C  \int_{t_0}^{t_1}{\Vert \tomega \hu  +\homega  \tu   \Vert_{L^2 (\Omega)}},
\end{equation}
which gives
\begin{equation}
\label{bdbd}
\Vert G[\, \tomega] \Vert_{\textrm{Lip}([0,T];H^1(\Omega))}\leq C \Vert \tomega \hu  +\homega  \tu   \Vert_{L^{\infty}([0,T];L^2(\Omega))}.
\end{equation}
Thus by mastering the right hand side of \eqref{bdbd} thanks to Proposition  \ref{p:yudo-ell-bound} and \eqref{xprime}, 
we conclude the proof of Lemma \ref{lemB}. 
\end{proof}

Using the decomposition \eqref{e:psi-decomp}, we deduce from  Lemma \ref{lemmaprime}, from
Lemma \ref{lemB} and from the $C^2  $ regularity on $\overline \Omega  $ of the 
the functions $f^i$
defined in \eqref{def-f},  
that  the stream function $\tilde{\psi}$ is in the space $H^1((0,T);H^1(\Omega))$.
The proof of 
Proposition \ref{prop-f}  is therefore completed.
\end{proof}

%%%%%%%%%%%%%%%%%%%%%%%%%%%%%%%%%%%%%%%%%%%%%%%%%%%%%%%%%%%%
\section{Energy estimate of the difference}
\label{p:ener-eq}

This section is devoted to the proof of the following energy inequality.
\begin{prop}\label{t:energ-eq}
For  $0 \leq t_0<t_1 \leq T$, 
\begin{equation} \label{eeq}
\frac{1}{2}\left[\Vert  \tu   \Vert_{L^2 (\Omega)}^2 \right]_{t_0}^{t_1}
    + \frac{1}{2}\int_{t_0}^{t_1}{\int_{\Gamma}{| \tu  |^2 g}}
    + \int_{t_0}^{t_1}{\int_{\Omega}{ \tu  \cdot \left(\left( \tu  \cdot \nabla\right)\hu \right)}} = 0,
\end{equation}
where we recall the notation \eqref{jump}.
\end{prop}

\begin{proof}
According to Proposition \ref{prop-f},
 we can apply \eqref{e:weak-trans-diff} 
 to  the case where the test function is the stream function $\tilde{\psi}$
 defined in \eqref{e:psi-decomp}. This entails that for $0 \leq t_0<t_1 \leq T$, 
\begin{equation} \label{nrj-base}
\int_{t_0}^{t_1} \int_{\Omega} \tomega\partial_t\tilde{\psi} +
\int_{t_0}^{t_1}{\int_{\Omega} \tomega \hu  \cdot  \nabla \tilde{\psi}}
+ \int_{t_0}^{t_1}{\int_{\Omega} \homega  \tu  \cdot  \nabla \tilde{\psi}}
    = \int_{t_0}^{t_1}{\int_{\Gamma}{\tilde{\psi} \tomega g}}
     +     \left[ \int_{\Omega}{\tomega\tilde{\psi}} \right]_{t_0}^{t_1}    .
\end{equation}
Let us successively transform  each term of \eqref{nrj-base}.

First, by \eqref{delta-hat} and an integration by parts, the first term in the left-hand side of \eqref{nrj-base} can be recast as 
\begin{align} \nonumber
\int_{t_0}^{t_1}{\int_{\Omega}{\tomega\partial_t\tilde{\psi}}}
    &= - \int_{t_0}^{t_1}{\int_{\Omega}{\nabla\tilde{\psi} \cdot \partial_t\nabla\tilde{\psi}}}
      + \int_{t_0}^{t_1}{\int_{\Gamma}{\pn\tilde{\psi}\, \partial_t\tilde{\psi}}} \\
    &= -\frac{1}{2}\left[\Vert  \tu   \Vert_{L^2 (\Omega)}^2 \right]_{t_0}^{t_1}
      + {\sum_{i\in I^*}{\tilde{\psi}'_i \, \int_{t_0}^{t_1} \tilde{\mathcal{C}}_i}} , \label{ICI1}
\end{align}
thanks to the boundary's decomposition \eqref{deco-do}, \eqref{e:u-nab-perp-psi},  \eqref{aubord}  and  \eqref{circu-psi-diff}.

Next, we recast  the second term in the left hand side of \eqref{nrj-base}  
 by using the following lemma. 

\begin{lem}\label{l:lemast-pre}
Let $u$  a divergence-free vector field in $C^0 (\overline \Omega; \R^2)$  with bounded  vorticity $\omega:= \curl u$.
Let $w$ a divergence-free  vector field in $H^1(\Omega)$.
Then 
\begin{align}
\label{VRAIAUSSI-pre}
\int_{\Omega}{\omega u^{\perp}\cdot w} 
 = - \frac12 \int_{\Gamma}{|u|^2(w\cdot n)}  - \int_{\Omega}{u\cdot \left(\left(u\cdot \nabla\right)w\right)} 
     + \int_{\Gamma}{(u\cdot w)(u\cdot n)} .
     % +\int_{\Omega}{|u|^2\div  \, (w)}      .
\end{align}
\end{lem}
We  use later on a technical lemma, see Lemma \ref{l:lemast},
 which slightly extends  
Lemma \ref{l:lemast-pre}. Indeed Lemma \ref{l:lemast-pre} is the particular case of 
Lemma \ref{l:lemast} below where $u=v$ and where $w$ is divergence-free, so that we do not provide here any proof of Lemma \ref{l:lemast-pre}.

Thus, with a preliminary recasting due to \eqref{e:u-nab-perp-psi}, 
\begin{equation}
\int_{t_0}^{t_1}{\int_{\Omega}{\tomega \hu  \cdot  \nabla \tilde{\psi}}}
    = -\int_{t_0}^{t_1}{\int_{\Omega}{\tomega  \tu  ^{\perp} \cdot  \hu }} %\\
    = \int_{t_0}^{t_1}{\left(\frac{1}{2}\int_{\Gamma}{| \tu  |^2 g} + \int_{\Omega}{ \tu  \cdot \left(\left( \tu  \cdot \nabla\right)\hu \right)}\right)} ,
    \label{ICI2}
\end{equation}
by applying Lemma \ref{l:lemast} with $(\tu   , \hu )$ instead of $(u,w)$, with  the observation that the last boundary term of \eqref{VRAIAUSSI-pre} vanishes thanks to 
\eqref{imp-diff}. Observe that we also use \eqref{imp-moy} to deal with the other  boundary term. 

Regarding the third term in the left hand side of \eqref{nrj-base}, it follows from 
\eqref{e:u-nab-perp-psi} that 
\begin{equation}
    \label{ICI3}
    \int_{t_0}^{t_1}{\int_{\Omega} \homega  \tu  \cdot  \nabla \tilde{\psi}} = 0.
\end{equation}

Using the decomposition \eqref{deco-do}, \eqref{aubord} and \eqref{noel}, we get that the first term in the right hand side of \eqref{nrj-base} satisfies
\begin{equation} \label{ICI4}
\int_{t_0}^{t_1}{\int_{\Gamma}{\tilde{\psi} \tomega g}}
    = -\int_{t_0}^{t_1}{\sum_{i\in I^*}{\tilde{\psi}_i \tilde{\mathcal{C}}_i'}}.
\end{equation}

Finally using again \eqref{delta-hat} and an integration by parts, 
we observe that the term involved in the bracket of the last term in the right hand side of \eqref{nrj-base} 
satisfies, for  every time, 
\begin{equation} 
\int_{\Omega}{\tomega\tilde{\psi}} 
    = -\int_{\Omega}{\nabla \tilde{\psi} \cdot \nabla\tilde{\psi}} + \int_{\Gamma}{\tilde{\psi}\, \pn\tilde{\psi}} \\
    = -\Vert  \tu   \Vert_{L^2 (\Omega)}^2 + \sum_{i\in I^*}{\tilde{\psi}_i \, \tilde{\mathcal{C}}_i},\label{ICI5}
\end{equation}
by using \eqref{e:u-nab-perp-psi}, the decomposition \eqref{deco-do} and \eqref{aubord}.

Combining \eqref{nrj-base}, \eqref{ICI1}, \eqref{ICI2}, \eqref{ICI3}, \eqref{ICI4} and
 \eqref{ICI5}, 
we arrive at \eqref{eeq} and this concludes the proof of Proposition \ref{t:energ-eq}.
\end{proof}

From Proposition \ref{p:ener-eq} we deduce the following estimate.
\begin{cor}\label{t:energ-eq-cor}
There exists a constant $C$ 
such that for any $0 \leq t_0<t_1 \leq T$,  and for every $p\in [2,+\infty)$, 
\begin{equation} \label{i:ener-ineq}
\frac{1}{2}\left[\Vert  \tu   \Vert_{L^2 (\Omega)}^2 \right]_{t_0}^{t_1}
+ \frac{1}{2}\int_{t_0}^{t_1}{\int_{\Go}{| \tu  |^2g}}
    \leq \frac{1}{2}\int_{t_0}^{t_1}{\int_{\Gi}{| \tu  |^2 (-g)}}
   + Cp \int_{t_0}^{t_1}{\Vert  \tu  \Vert_{L^2(\Omega)}^{2\frac{p-1}{p}}}.
\end{equation}
\end{cor}
Observe that due to the sign conditions in \eqref{signeg}, the second term in the left  hand side and the first term in the right hand side are both non negative; 
this is indeed the reason why we display them in this way.
Should the latter be discarded, the inequality \eqref{i:ener-ineq} would allow to conclude  the proof of Theorem \ref{main-thm} by proceeding as in the impermeable case, that is by optimizing  in $p$ and by using Osgood's lemma, see Remark  \ref{rk-c-p} and \cite{yudo1995,yudo2005}.
As already hinted above, see Section  \ref{sec-st}, the presence of this "bad" boundary term, which is associated with the sources, but not bounded by the left hand side nor by the data, is an obstacle to this program, which is overcome in the next sections, by making use of a well-chosen auxiliary test function.
\begin{proof}[Proof of Corollary \ref{t:energ-eq-cor}]
Taking into account the splitting of the second term of 
\eqref{t:energ-eq} according to the decomposition of the boundary $\Gamma $ into 
$\Gamma = \Gi \cup \Go \cup \Gamma_0$ and the fact that 
$g = 0$  on $ \Gamma_0$ (see  \eqref{deco-do}, 
 \eqref{tison} and \eqref{signeg}), 
 the proof of Corollary \ref{t:energ-eq-cor} amounts to prove that 
 there exists a constant $C$ 
such that at any non-negative time,  and for every $p\in [2,+\infty)$, 
 the third term of \eqref{eeq} is bounded as follows: 
\begin{equation}  \label{oeuf}
\left|\int_{\Omega}{ \tu  \cdot \left(\left( \tu  \cdot \nabla\right)\hu \right)}\right|
    \leq  C  p \Vert  \tu  \Vert_{L^2(\Omega)}^{2\frac{p-1}{p}} . 
\end{equation}
To that aim, 
let us first observe that, by H\"older's inequality, we have, for all $p\in (1,+\infty)$, that 
\begin{equation} \label{paques}
\left|\int_{\Omega}{ \tu  \cdot \left(\left( \tu  \cdot \nabla\right)\hu \right)}\right|
    \leq \Vert  \tu  \Vert_{L^{\frac{2p}{p-1}}(\Omega)}^2 \Vert \hu \Vert_{W^{1,p}(\Omega)}, 
\end{equation}
To bound the first factor in the right hand side, let us recall  that for any $p$ in $(1,+\infty)$, for any vector field $w$
 in $L^2 ( \Omega; \R^2) \cap L^\infty ( \Omega; \R^2)$,
 we have the following interpolation inequality: 
\begin{equation} \label{interpo}
\Vert w\Vert_{L^{\frac{2p}{p-1}  }(\Omega)}
    \leq \Vert w\Vert_{L^{\infty} (\Omega) }^{\frac{1}{p}}
     \Vert w\Vert_{L^2  (\Omega)}^{\frac{p-1}{p}} . 
\end{equation}
By Proposition \ref{prop-f}, Sobolev embedding theorem and \eqref{e:u-nab-perp-psi}, we obtain that 
 $\tu$ is in the space $C^0 (\R_+ ; L^{\infty} (\Omega))$, so that the first factor of  the right hand side of \eqref{paques} can be bounded by 
\begin{equation*} 
 C   \Vert  \tu \Vert_{L^2  (\Omega)}^{2\frac{p-1}{p}} . 
\end{equation*}
Moreover the second factor of  the right hand side of \eqref{paques} can be bounded thanks to  Proposition \ref{p:yudo-ell-bound} by $Cp$ for a positive constant  $C$ independent of time and of $p$ in $[2,+\infty)$. 
This allows to deduce  \eqref{oeuf} from \eqref{paques}, and therefore to conclude the proof of Corollary~\ref{t:energ-eq-cor}.
\end{proof}
%

%%%%%%%%%%%%%%%%%%%%%%%%%%%%%%%%%%%%%%%%%%%%%%%%%%%%%%%%%%%%
\section{An auxiliary test function}
\label{s:2}

This section is devoted to the existence and to the regularity of an auxiliary test function which is useful in the sequel to establish a second energy estimate.
This test function is defined through the following Zaremba-type problem:
\begin{subequations}
\label{caremba}
\begin{align}
\label{lap}
\Delta \tilde{\varphi}  &= 0             \quad                                  \text{ on } \Omega, \\ \label{lap2}
\tilde{\varphi}         &= 0                \quad                                 \text{ on } \Go\cup\Gamma_0 ,  \\ \label{lap3}
\pn \tilde{\varphi}     &= -\pn \tilde{\psi}   \quad                                             \text{ on } \Gi.
\end{align}
\end{subequations}
It is instructive to look first at a variational formulation of the system \eqref{caremba}, associated with the space 
\begin{equation}\label{e:def-Ho}
\Ho:=\{a\in H^1(\Omega); \quad a_{|\Go\cup\Gamma_0}=0\}.    
\end{equation}
Thanks to the Poincar\'e inequality, the space $\Ho$ is a Hilbert space endowed with the homogeneous Sobolev norm associated with the space $\dot{H}^1(\Omega)$. 
Then, we consider the following variational formulation of the system \eqref{caremba}:
\begin{equation}\label{e:zarem-weaksense}
 \tilde{\varphi} \text{ is in } \Ho \text{ and }
\forall a\in \Ho, \quad
\int_{\Omega}{\nabla\tilde{\varphi}\cdot\nabla a}
    = - \int_{\Omega}{\nabla\tilde{\psi}\cdot\nabla a}
     - \int_{\Omega}{\tomega a}.
\end{equation}
Indeed, if $\tilde{\varphi}$ is a smooth solution of the system \eqref{caremba}, then for all $a$ in  $\Ho$,
\begin{equation}
\label{car1}
0 = \int_{\Omega} (\Delta \tilde{\varphi} ) a 
= -\int_{\Omega} \nabla  \tilde{\varphi} \cdot \nabla a
+ \int_{\Gamma} (\pn  \tilde{\varphi}  )a,
\end{equation}
by integrating by parts. Moreover by \eqref{lap3} and the fact that $a\in \Ho$, we deduce that 
\begin{equation}
\label{car2}
 \int_{\Gamma} (\pn  \tilde{\varphi}  )a 
 = - \int_{\Gamma} (\pn  \tilde{\psi}  )a
  = - \int_{\Gamma} \nabla  \tilde{\psi} \cdot \nabla a
  -  \int_{\Omega} (\Delta \tilde{\psi} ) a  ,
\end{equation}
by integrating by parts. Combining \eqref{car1}, \eqref{car2} and recalling 
\eqref{delta-hat}, we arrive at \eqref{e:zarem-weaksense}.

Since the right hand side of the identity in \eqref{e:zarem-weaksense} is a linear form with respect to $a$ in $\Ho$ with an operator norm bounded by
\begin{equation}
\label{lactou}
C \big( \Vert\tu \Vert_{C^0([0,T];L^2(\Omega))}
+ \Vert\tomega \Vert_{C^0([0,T];L^2(\Omega))} \big),
\end{equation} 
by the Lax-Milgram theorem, there exists a unique function $\tilde{\varphi}$ in $C^0([0,T];\Ho)$ which satisfies \eqref{e:zarem-weaksense} for all non-negative times. 

Moreover the following result establishes that the function $\tilde{\varphi}$ is more regular. 
\begin{prop}\label{p:aux-test-reg}
The function $\tilde{\varphi}$ is  in $C^0([0,T];W^{2,p}(\Omega))$ for all $p$ in $(1,+\infty)$. 
Moreover, for all $T>0$, the function $\tilde{\varphi}$ is also in the space
$H^1([0,T];\Ho)$.
%$H^1([0,T];H^1(\Omega))$. 
\end{prop}
\begin{proof}
First, by Proposition \ref{p:yudo-ell-bound}, we know that $\tilde{\psi}$ is in $C^0([0,T];W^{2,p}(\Omega))$ for any finite~$p$. 
Then, it follows from the classical regularity results on the Zaremba-type problems,
see for example \cite[Theorem 2.4.2.6]{Grisvard},
that the function $\tilde{\varphi}$ is also in $C^0([0,T];W^{2,p}(\Omega))$ for any $p$ in $(1,+\infty)$.

Thus, it only remains to prove that $\partial_t \tilde{\varphi}$ is 
 in the space $L^2((0,T); \Ho)$. 
To this aim, let us consider $\omega_\phi$ a smooth function on $[0,T]\times\Omega$, compactly supported in $(0,T)$.
As argued above, there exists a unique smooth function $\phi$ which satisfies for all $t$ in $[0,T]$,
\begin{subequations}
\label{tes}
\begin{align}  \label{tes1}
\Delta \phi   &=   \omega_\phi            \quad                   \text{ on } \Omega, \\ \label{tes2}
\phi          &= 0                        \quad                   \text{ on } \Go\cup\Gamma_0 ,  \\ \label{tes3}
\pn \phi      &= 0                        \quad                   \text{ on } \Gi.
\end{align}
\end{subequations}
Moreover, the function  $\phi$ is  compactly supported in $(0,T)$.
Let us denote by $H'$ the dual space of $H^1(\Omega)$ for the $L^2$ scalar product, seen as a subset of $H^{-1}(\Omega)$.
Using the fact that we have the equality 
\begin{equation}
\int_{\Omega}{\phi\omega_{\phi}} = -\int_{\Omega}{|\nabla\phi|^2},
\end{equation}
we get that the function $\phi$ is also in $L^{\infty}([0,T];H^1(\Omega))$, together with the estimate 
\begin{equation}\label{e:phi-omephi-est}
\Vert\phi(t,.)\Vert_{H^1(\Omega)}\leq C \Vert\omega_{\phi}(t,.)\Vert_{H'}, 
\end{equation}
for every $t$ in $[0,T]$ 
with some constant $C$ independent of the time.

By integrating by parts, and using that on the one hand 
$\partial_t \pn \phi     =\pn  \partial_t  \phi  = 0  $  on $ \Gi$ for all $t$ in $[0,T]$, 
and on the other hand \eqref{lap2}, we have that 
\begin{equation}
    \label{dua1}
    \int_\Omega \tilde{\varphi} \partial_t \omega_\phi 
    = -  \int_\Omega \nabla  \tilde{\varphi} \cdot  \nabla \partial_t \phi .
\end{equation}
Since, for  all $t$ in $[0,T]$, the function  $\partial_t \phi (t,\cdot)$ is in 
$\Ho$, we have, by applying \eqref{e:zarem-weaksense} with $\partial_t \phi (t,\cdot)$ instead of $a$, and integrating over $(0,T)$,  that 
\begin{equation}
\label{dua2}
  \int_0^T    \int_\Omega \nabla  \tilde{\varphi} \cdot \nabla \partial_t \phi 
    = -  \int_0^T   \int_\Omega \nabla  \tilde{\psi} \cdot \nabla \partial_t \phi     -  \int_0^T  \int_\Omega \tomega \partial_t \phi .
\end{equation}
By integrating by parts in space and in time and using  \eqref{tes1}, we deduce that 
\begin{align} \nonumber
   \int_0^T   \int_\Omega \nabla  \tilde{\psi} \cdot \nabla \partial_t \phi
   &= -  \int_0^T  \int_\Omega  \tilde{\psi} \partial_t \omega_\phi
   +  \int_0^T  \int_\Gamma  \tilde{\psi} \pn \partial_t \phi 
 \\ \label{erc} &=    \int_0^T  \int_\Omega (\partial_t \tilde{\psi})  \omega_\phi
   -  \int_0^T  \int_\Gamma  (\partial_t \tilde{\psi}) \pn  \phi .
   \end{align}
On the one hand, by the Cauchy-Schwarz inequality 
and Proposition \ref{prop-f}, 
\begin{equation}
\label{inter1}
\left\vert \int_0^T  \int_\Omega (\partial_t \tilde{\psi})  \omega_\phi \right\vert 
    \leq C \Vert\omega_\phi \Vert_{L^{2}((0,T);H')}.
\end{equation}
On the other hand, by using the boundary decomposition \eqref{deco-do} and recalling  \eqref{aubord}, 
\begin{equation}
\label{inter2}
 \int_0^T  \int_\Gamma  (\partial_t \tilde{\psi}) \pn  \phi =
 \sum_{i\in I}  \int_0^T   \tilde{\psi}_i' \int_{\Gamma_i}  \pn  \phi  .
  \end{equation}
Moreover, for any $i$ in $I^*$, recalling the definition of $f^i$ in \eqref{def-f}, we have, for all non negative times, that
\begin{align*}
 \int_{\Gamma_i}  \pn  \phi 
 = \int_{\Omega} \nabla f^i \cdot \nabla  \phi 
 +  \int_{\Omega}  f^i \omega_\phi ,
\end{align*}
so that, by the Cauchy-Schwarz inequality and \eqref{e:phi-omephi-est}, we deduce that 
\begin{align*}
\left\vert \int_{\Gamma_i}  \pn  \phi\, \right\vert
    \leqslant  C \Vert\omega_\phi \Vert_{H'} .
\end{align*}
Therefore, by \eqref{inter2}, the Cauchy-Schwarz inequality with respect to time, Proposition \ref{prop-f} and recalling that $\tilde{\psi}_0$ is identically null, we obtain that 
\begin{equation}
\label{inter4}
\left\vert  \int_0^T  \int_\Gamma  (\partial_t \tilde{\psi}) \pn  \phi\, \right\vert
    \leq C \Vert\omega_\phi \Vert_{L^{2}((0,T);H')} .
\end{equation}
Combining \eqref{erc}, \eqref{inter1} and \eqref{inter4}, 
we infer  that 
\begin{equation}
\label{dua3}
\left\vert \int_0^T   \int_\Omega \nabla  \tilde{\psi} \cdot \nabla \partial_t \phi\, \right\vert
    \leq C  \Vert\omega_\phi \Vert_{L^{2}((0,T);H')} .
\end{equation}
We now turn to the estimate of the last term of \eqref{dua2}.
By \eqref{e:weak-trans-diff}, with $(t_0 , t_1 ) =(0,T)$, 
\begin{equation*}
- \int_{0}^{T} \int_{\Omega} \tomega\partial_t \phi 
= \int_{0}^{T} \int_{\Omega} 
\big( \tomega \hu \cdot \nabla \phi  +\homega \tu  \cdot  \nabla \phi \big) 
    - \int_{0}^{T}{\int_{\Gamma}{\phi \tomega g}} .
\end{equation*}
Since $\tomega$ and $\homega$ are in $L^\infty ((0,T) \times \Omega) $,
since  $\tu$ and $\hu$ are in $L^2 ((0,T) \times \Omega) $, 
since $\tomega$ is in  $L^{\infty}(0,T ;L^{\infty}(\Go,  |g|))$, by the classical trace properties, see for example \cite[Section 13.6]{TW}, 
and using \eqref{e:phi-omephi-est},
we obtain that 
\begin{equation}\label{dua4}
\left\vert \int_{0}^{T} \int_{\Omega} \tomega\partial_t \phi \right\vert 
    \leq C  \Vert\omega_\phi \Vert_{L^{2}((0,T);H')} .
\end{equation}
By gathering \eqref{dua1}, \eqref{dua2}, \eqref{dua3}  and \eqref{dua4}, we deduce that 
for any  function  $\omega_\phi$ which is smooth on $[0,T]\times\Omega$ and compactly supported in $(0,T)$, 
\begin{equation*}
\left\vert \int_{0}^{T}  \int_\Omega \tilde{\varphi} \partial_t \omega_\phi\, \right\vert
   \leq C  \Vert\omega_\phi \Vert_{L^{2}((0,T);H')} .
\end{equation*}
This entails that  $\partial_t \tilde{\varphi}$ is  in the space $L^2((0,T);H^1 (\Omega))$ 
and it follows from the properties of the trace that  $\partial_t \tilde{\varphi}$ is indeed
 in the space $L^2((0,T); \Ho)$. This concludes the proof of Proposition \ref{p:aux-test-reg}. 
\end{proof}

%%%%%%%%%%%%%%%%%%%%%%%%%%%%%%%%%%%%%%%%%%%%%%%%%%%%%%%%%%%%
\section{A generalized Lamb-type lemma}
\label{Lamb-G}

This section is devoted to the following Lamb-type lemma.
\begin{lem}\label{l:lemast}
Let $u,v$ two divergence-free vector fields in 
$C^0 (\overline \Omega; \R^2)$   with bounded  vorticity.
Let $w$ a vector field in $H^1(\Omega ; \R^2)$.
We have the following equality:

\begin{align}
\label{VRAI}
\int_{\Gamma}{(u\cdot v)(w\cdot n)}
    &-\int_{\Gamma}{(u\cdot w)(v\cdot n)}
    -\int_{\Gamma}{(v\cdot w)(u\cdot n)} \\ \nonumber
     &=\int_{\Omega}{(u\cdot v)\div w}
    -\int_{\Omega}{(\curl u)\, v^{\perp}\cdot w}
    -\int_{\Omega}{(\curl v)\, u^{\perp}\cdot w} \\ \nonumber
    &-\int_{\Omega}{u\cdot \left(\left(v\cdot \nabla\right)w\right)}
    -\int_{\Omega}{v\cdot \left(\left(u\cdot \nabla\right)w\right)} .
\end{align}
 In particular, in the case where $u=v$, then \eqref{VRAI} reduces to 
\begin{align}
\label{VRAIAUSSI}
\int_{\Gamma}{|u|^2(w\cdot n)}
    &-2\int_{\Gamma}{(u\cdot w)(u\cdot n)} \\ \nonumber
     &=\int_{\Omega}{|u|^2 \div  \, w}
    -2\int_{\Omega}{\omega u^{\perp}\cdot w}
    -2\int_{\Omega}{u\cdot \left(\left(u\cdot \nabla\right)w\right)} ,
\end{align}
where $\omega:= \curl \, u$.
\end{lem}
Lemma \ref{l:lemast} implies 
 Lemma \ref{l:lemast-pre} by considering the particular case where $u=v$ and where $w$ is divergence free. It also implies \cite[Lemma 7]{bd}  and
\cite[Lemma 2.15]{Prague} in the particular cases where $w$ is a rigid velocity.  
It is also  useful several times in the sequel.
\begin{proof}
It is sufficient to prove \eqref{VRAI} in the case where the three vector fields $u$, $v$ and $w$ are smooth, since then the result follows from an approximation process. 
We start with using Stokes' formula to obtain that
\begin{align} \label{VRAI1}
\int_{\Gamma}{(u\cdot v)(w\cdot n)} 
    = \int_{\Omega}{\div((u\cdot v)w)} 
    = \int_{\Omega}{\nabla(u\cdot v)\cdot w} + \int_{\Omega}{(u\cdot v)\div w} .
\end{align}
Moreover,
\begin{equation}
\nabla(u\cdot v)
    = (u\cdot \nabla)v + (v\cdot \nabla)u
     - (\curl u )\, v^{\perp} - (\curl v)\, u^{\perp} ,
\end{equation}
and therefore 
\begin{align} \label{VRAI2}
\int_{\Gamma}{(u\cdot v)(w\cdot n)}
     =&\int_{\Omega}{(u\cdot v)\div  \,w}
    -\int_{\Omega}{(\curl u)\, v^{\perp}\cdot w}
    -\int_{\Omega}{(\curl v)\, u^{\perp}\cdot w} \\ \nonumber 
    & + \int_{\Omega}  {((u\cdot \nabla)v) \cdot w}
    +\int_{\Omega} {((v\cdot \nabla)u) \cdot w} .
\end{align}
Finally, by some  integrations by parts,  since  $u$ and $v$ are divergence-free,  
\begin{align} \label{VRAI3}
\int_{\Omega}{((u\cdot \nabla)v) \cdot w}
    &= -\int_{\Omega}{v \cdot ((u\cdot \nabla)w)}
     + \int_{\Omega}{(v\cdot w)(u\cdot n)}, \quad \text{and}
     \\ \nonumber \int_{\Gamma}{((v\cdot \nabla)u) \cdot w}
    &= -\int_{\Omega}{u \cdot ((v\cdot \nabla)w)}
     + \int_{\Gamma}{(u\cdot w)(v\cdot n)} .
\end{align}
Gathering 
\eqref{VRAI1}, \eqref{VRAI2} and \eqref{VRAI3}, we obtain \eqref{VRAI}.
\end{proof}

%%%%%%%%%%%%%%%%%%%%%%%%%%%%%%%%%%%%%%%%%%%%%%%%%%%%%%%%%%%%
\section{An auxiliary energy-type estimate}
\label{sec-aux-esti}

Recall that the function  $\tilde{\varphi}$ is defined in Section \ref{s:2}.
We denote by $\tv  $ its orthogonal~gradient
\begin{equation} \label{hatv}
\tv  := \nabla^{\perp} \tilde{\varphi} .
\end{equation}
We also set, 
for $i$ in $I^*$, 
\begin{equation}\label{def-Di}
\tilde{\mathcal{D}}_i 
    := \int_{\Gamma_i}{\pn\tilde{\varphi}} .
\end{equation}
Let us observe that,  
it follows from \eqref{lap3}  that, at any time $t\in [0,T]$, 
\begin{equation} \label{rgf}
\tilde{\mathcal{D}}_i (t)= -\tilde{\mathcal{C}}_i (t)    \quad      \text{ for all  }  i\in I_{in}.
\end{equation}
\begin{prop} \label{prop-auxesti}
For   $0\leq t_0<t_1 \leq T$, 
\begin{align} \label{auxe}
\frac{1}{2}\left[\Vert \tv   (t,.) \Vert_{L^2 (\Omega)}^2\right]_{t_0}^{t_1}
     &+\int_{t_0}^{t_1}{\int_{\Gi}{| \tu  |^2(- g)}} \\ \nonumber
    &= \int_{t_0}^{t_1}{\int_{\Go}{( \tu  \cdot \tv   )(- g)}} 
      + \int_{t_0}^{t_1}{\int_{\Gi}{( \tu  \cdot\hu )(\tv   \cdot n)}} \\ \nonumber
     &\quad + \int_{t_0}^{t_1}\int_{\Omega} 
          \Big( \tu  \cdot ((\tv   \cdot\nabla)\hu )
          +\tv   \cdot (( \tu  \cdot\nabla)\hu )
          +{\homega  \tu  \cdot \tv   ^{\perp}} \Big)  \\ \nonumber
     &\quad - \int_{t_0}^{t_1}{\sum_{i\in I^*}{\tilde{\psi}_i' \tilde{\mathcal{D}}_i}} 
      + \int_{t_0}^{t_1}{\int_{\Gi}{\tilde{\varphi} \tomein  g}} . 
\end{align}
\end{prop}
Observe that due to the sign conditions in \eqref{signeg}, the second term in the left  hand side is non negative; and is twice the  "bad" boundary term in 
 \eqref{i:ener-ineq}.  
\begin{proof}
Thanks to Proposition \ref{p:aux-test-reg},
we can apply \eqref{e:weak-trans-diff}  to  the case
where the test function is the function  $\tilde{\varphi}$.
This entails that for  every $t_0,t_1$
\begin{equation}
\label{anum}
\int_{t_0}^{t_1} \int_{\Omega} \tomega\partial_t\tilde{\varphi} 
+ \int_{t_0}^{t_1} \int_{\Omega} \tomega \hu  \cdot  \nabla \tilde{\varphi}
+ \int_{t_0}^{t_1} \int_{\Omega}  \homega  \tu   \cdot  \nabla \tilde{\varphi}
    = \int_{t_0}^{t_1}{\int_{\Gamma}{\tilde{\varphi} \tomega g}}
    +\left\lbrack \int_{\Omega}{\tomega\tilde{\varphi}}  \right\rbrack_{t_0}^{t_1} .
\end{equation}
Let us simplify each term of \eqref{anum}.

Let us start with the last one.
Taking $\tilde{\varphi}$ as a test function in \eqref{e:zarem-weaksense}, and  using  \eqref{hatv}, we get that:
\begin{equation}
\label{ajou1}
-\int_{\Omega}{|\tv|^2}= \int_{\Omega} {\nabla\tilde{\psi}\cdot\nabla\tilde{\varphi}} + \int_{\Omega}{\tomega \tilde{\varphi}} .
\end{equation}
We can simplify the first term of the right hand side of \eqref{ajou1} by an integration by parts:
\begin{equation}
\label{ajou2}
 \int_{\Omega}{\nabla\tilde{\psi}\cdot\nabla\tilde{\varphi}} 
    = - \int_{\Omega}{\tilde{\psi}\Delta\tilde{\varphi}}
      + \int_{\Gamma}{\tilde{\psi}\pn\tilde{\varphi}}
    = \sum_{i\in I^*}{\tilde{\psi}_i\tilde{\mathcal{D}}_i} ,
\end{equation}
by \eqref{lap}, \eqref{deco-do}, 
 \eqref{aubord} and  \eqref{def-Di}.
Therefore, by combining \eqref{ajou1} and \eqref{ajou2},
we obtain that the last term of \eqref{anum} can be recast as
\begin{equation} \label{lourd}
\int_{\Omega}{\tomega\tilde{\varphi}} = -\int_{\Omega}{|\tv|^2} - \sum_{i\in I^*}{\tilde{\psi}_i\tilde{\mathcal{D}}_i} .
\end{equation}

We now turn to the the first term in the left hand side of \eqref{anum} for which we apply  similar computations. More precisely,  by taking $\partial_t  \tilde{\phi}$ as a test function in \eqref{e:zarem-weaksense},  and  using again \eqref{hatv}, we obtain that 
\begin{equation}  \label{balai1}
-\int_{\Omega}{\tv\cdot \partial_t\tv}= 
 \int_{\Omega}{\nabla\tilde{\psi}\cdot\partial_t\nabla\tilde{\varphi}} 
 + \int_{\Omega}{\tomega \partial_t \tilde{\varphi}} .
\end{equation}
Moreover, by an integration by parts, the first term in the right hand side of  \eqref{balai1} can be simplified into 
\begin{equation}  \label{balai2}
  \int_{\Omega}{\nabla\tilde{\psi}\cdot\partial_t\nabla\tilde{\varphi}} 
    = - \int_{\Omega}{\tilde{\psi}\partial_t\Delta\tilde{\varphi}}
      + \int_{\Gamma}{\tilde{\psi}\partial_t\pn\tilde{\varphi}} \\
   = \sum_{i\in I^*}{\tilde{\psi}_i\tilde{\mathcal{D}}_i'} .
\end{equation}
By combining  \eqref{balai1} and  \eqref{balai2} 
and integrating between $t_0$ and $t_1$, we obtain:
\begin{equation} \label{leadus}
\int_{t_0}^{t_1}{\int_{\Omega}{\tomega\partial_t\tilde{\varphi}}}
    = -\frac{1}{2}\left[\Vert \tv \Vert_{L^2  (\Omega)}^2\right]_{t_0}^{t_1} 
    - \int_{t_0}^{t_1}{\sum_{i\in I^*}{\tilde{\psi}_i \tilde{\mathcal{D}}_i'}} .
\end{equation}

The second term in the left hand side of \eqref{anum} can be tackled by 
 Lemma \ref{l:lemast}  to $(\hu ,  \tv   , \hu )$ instead of 
$(u,v,w)$. Using also that  $\curl \tv   =0$ in $\Omega$ 
and \eqref{imp-diff}  this entails that 
\begin{align*}
\int_{\Omega}{\tomega \hu  \cdot  \nabla \tilde{\varphi}}
    &= \int_{\Gamma}{( \tu  \cdot \tv   ) g} - \int_{\Gamma}{( \tu  \cdot \hu )(\tv   \cdot n)} 
    + \int_{\Omega}{ \tu  \cdot ((\tv   \cdot\nabla)\hu )} + \int_{\Omega}{\tv   \cdot (( \tu  \cdot\nabla)\hu )} .
\end{align*}
Moreover, using \eqref{lap3}, 
\eqref{hatv} and again \eqref{imp-diff}, 
the first term in the left hand side above can be decomposed as follows
\begin{equation}
\int_{\Gamma}{( \tu  \cdot \tv   ) g}
    = \int_{\Go}{ (\tu  \cdot \tv ) \,  g}
     + \int_{\Gi}{|\tu  |^2 (-g)},
\end{equation}
and the second one can be simplified into
\begin{equation}
\int_{\Gamma}{( \tu  \cdot \hu )(\tv   \cdot n)} 
=
\int_{\Gi}{( \tu  \cdot \hu )(\tv   \cdot n)} ,
\end{equation}
by  \eqref{hatv} and \eqref{lap2}.
We therefore arrive at  
\begin{align} 
\nonumber
\int_{t_0}^{t_1} \int_{\Omega} \tomega \hu  \cdot  \nabla \tilde{\varphi}
    &= \int_{t_0}^{t_1} \int_{\Go}{( \tu  \cdot \tv  ) g}
     +\int_{t_0}^{t_1} \int_{\Gi}{|\tu  |^2 (-g)} 
    - \int_{\Gi}{( \tu  \cdot \hu )(\tv   \cdot n)}\\
    % - \int_{t_0}^{t_1} \int_{\Go}{( \tu  \cdot\tau) (\tv   \cdot \tau) g} \\
     &+\int_{t_0}^{t_1} \int_{\Omega} \Big(  \tu  \cdot ((\tv   \cdot\nabla)\hu ) + \tv   \cdot (( \tu  \cdot\nabla)\hu ) \Big) . \label{deja}
\end{align}
The third term of the right hand side of \eqref{anum} satisfies
\begin{equation}
\label{batar}
\int_{t_0}^{t_1} \int_{\Omega}  \homega  \tu   \cdot  \nabla \tilde{\varphi} 
   = -\int_{t_0}^{t_1}\int_{\Omega} {\homega  \tu  \cdot \tv   ^{\perp}} ,
\end{equation}
by \eqref{hatv}. 

Finally the first term in the left hand side of \eqref{anum} satisfies
\begin{equation}
\label{marre}
\int_{t_0}^{t_1}{\int_{\Gamma}{\tilde{\varphi} \tomega  g}}
    = \int_{t_0}^{t_1}{\int_{\Gi}{\tilde{\varphi} \tomein  g}},
\end{equation}
by  \eqref{lap2}.

Gathering  \eqref{anum}, \eqref{lourd}, \eqref{leadus}, \eqref{deja}, \eqref{batar}  and \eqref{marre},
we arrive at \eqref{auxe}.
This concludes the proof of Proposition \ref{prop-auxesti}.
\end{proof}

In the same way as we deduce Corollary \ref{t:energ-eq-cor} from Proposition \ref{t:energ-eq},
we deduce the following corollary from Proposition \ref{prop-auxesti}. 
\begin{cor}\label{auxi-cor}
There exists a constant $C>0$ such that for every $p$ in $[2,+\infty)$
and  for every $0\leq t_0<t_1 \leq T$,
\begin{align}
\nonumber
\frac{1}{2} \left[\Vert \tv   (t,.) \Vert_{L^2  (\Omega)}^2\right]_{t_0}^{t_1}
     & + \frac{7}{8} \int_{t_0}^{t_1}{\int_{\Gi}{| \tu  |^2 (- g)}} 
\nonumber
    \leq   \frac{1}{4}  \int_{t_0}^{t_1}{\int_{\Go}{| \tu  |^2 (- g)}}
  \\
\nonumber  & + C\int_{t_0}^{t_1}{\left(\Vert \tv    \Vert_{L^2(\Omega)}^2
       +\Vert\tu\Vert_{L^2(\Omega)}^2\right)}
       + Cp \int_{t_0}^{t_1}{\left(
                \Vert\tv\Vert_{L^{2}(\Omega)}^{2\frac{p-1}{p}} 
                 + \Vert  \tu  \Vert_{L^{2}(\Omega)}^{2\frac{p-1}{p}}\right)}  \\
\label{i:aux-ineq}
     & + C(t_1-t_0).\Vert\tomein\Vert_{L^{\infty}([0,t];L^{\infty}(\Gi,|g|))}^2    
   + (t_1-t_0)\sum_{i\in I_{\mathrm{in}}}{ |\tilde{\mathcal{C}}_{i,0}|^2}    . 
\end{align}
\end{cor}
In the proof of Corollary  \ref{auxi-cor}
we use the following result on the trace of  harmonic functions. 
\begin{lem}\label{l:horn-ineq}
For all $i\in I$, there exists a constant $C>0$ such that for any function $h$  which is  harmonic on  $\Omega$ and $C^1$ on $\overline{\Omega}$, 
\begin{equation} \label{rafol}
\int_{\Gamma_i}{\left(\pn h\right)^2}
    \leq \int_{\Gamma_i}{\left(\partial_{\tau}h\right)^2}
      +C \int_{\Omega}{|\nabla h|^2} .
\end{equation}
\end{lem}
Lemma \ref{l:horn-ineq} can be seen for example as local version of H\"ormander's trace inequality (see \cite{Horn:traineq}). Let us refer here to \cite{GKLP} for more on the historical context of this trace inequality 
and its relationship with Rellich and Pohozaev types inequalities. 
For sake of completeness we provide below a proof which uses Lemma \ref{l:lemast}.

\begin{proof}[Proof of Lemma \ref{l:horn-ineq}] 
For all $i\in I$, there is  a  $C^1$ vector field $N_i: \overline{\Omega}\longrightarrow \R^2$ which is  equal to the normal vector  $n$ on $\Gamma_i$ and which is equal to $0$ in a neighbourhood of $\Gamma\setminus\Gamma_i$.
We apply  \eqref{VRAIAUSSI} with $( \nabla^{\perp} h , N_i )$ instead of  $(u,w)$.
Observe that the first vector field is divergence free since the function $h$ is harmonic. 
This entails that 
\begin{equation} \label{fol}
\int_{\Gamma_i}{|\nabla h|^2} 
 - 2\int_{\Gamma_i}{\left(\partial_{\tau}h\right)^2}
    = \int_{\Omega}{|\nabla h|^2 (\div N_i)} 
      - 2\int_{\Omega}{\nabla^{\perp} h\cdot\left(\left(\nabla^{\perp} h\cdot \nabla\right) N_i\right)}.
\end{equation}
Since $N_i$ is $C^1$, we deduce from \eqref{fol} that   there exists a constant $C>0$ such that for any function $h$ that is  harmonic on  $\Omega$ and $C^1$ on $\overline{\Omega}$, the following estimate holds true 
\begin{equation}
\left|\int_{\Gamma_i}{\left(\pn h\right)^2-\left(\partial_{\tau}h\right)^2}\right|
    \leq C \int_{\Omega}{|\nabla h|^2} ,
\end{equation}
and thus in particular \eqref{rafol}.
\end{proof}

\begin{proof}[Proof of Corollary  \ref{auxi-cor}] 
We successively bound the terms in the right hand side of \eqref{auxe}.
The first term can be estimated as follows.
\begin{lem}\label{l:ineq-go}
For all $\varepsilon>0$, there exists $M(\varepsilon)>0$ such that 
%[2,+\infty)
\begin{equation}
\label{evla1}
\left|\int_{\Go} {\left( \tu  \cdot\tv    \right) g} \right|
    \leq \varepsilon \int_{\Go}{| \tu  |^2g } 
     + M(\varepsilon)\int_{\Omega}{|\tv   |^2} .
\end{equation}
\end{lem}
\begin{proof}[Proof of Lemma \ref{l:ineq-go}]
First, since $\tu   \cdot n =0$, 
\begin{equation}
\int_{\Go}{\left( \tu  \cdot \tv    \right) g}
    = \int_{\Go}{\left( \tu  \cdot \tau\right)\left(\tv   \cdot \tau\right) g} .
\end{equation}

Using \eqref{hatv} 
and  the fact that $ g$ is bounded on $\Gamma$ we deduce that for  all $\varepsilon>0$, there exists $\tilde{M}(\varepsilon)>0$ such that 
\begin{align} \label{ewa1}
\left|\int_{\Go} {\left( \tu  \cdot\tv    \right) g} \right|
&\leq \varepsilon\int_{\Go}{| \tu  |^2\left|g\right|}+\tilde{M}(\varepsilon)
    \int_{\Go}{|\pn \tilde{\varphi}|^2} . 
\end{align}
Now we apply  Lemma \ref{l:horn-ineq}  
for $i$ in $I_{\mathrm{out}}$ with $\tilde{\varphi}$ instead of $h$ and we use that $\tilde{\varphi}$ is equal to $0$ on $\Go$
to obtain that there exists $C>0$ such that 
\begin{equation} \label{hoapplied}
     \int_{\Go}{|\pn \tilde{\varphi}|^2}
     \leq C \int_{\Omega}{|\tilde{v}|^2} .
\end{equation}
Combining  \eqref{ewa1} and  \eqref{hoapplied} and setting $M(\varepsilon):= C \tilde{M}(\varepsilon)$,
we conclude the proof of Lemma~\ref{l:ineq-go}.
\end{proof}

The second term in the right hand side of \eqref{auxe} can be estimated as follows.
\begin{lem}\label{l:ineq-gi}
There exists a constant $C>0$ such that for all $p$ in $[2,+\infty)$, 
\begin{equation} \label{evla2}
\left|\int_{\Gi}{\left(  \hu \cdot \tu   \right)\left(\tv   \cdot n \right)}\right|
    \leq Cp \Vert \tv   \Vert_{L^{2}(\Omega)}^{2\frac{p-1}{p}} .
\end{equation}
\end{lem}

\begin{proof}
First, since $\tu   \cdot n =0$ 
and $ \tu  \cdot \tau = - \tv   \cdot \tau $ on $\Gi$, we have:
\[ 
\int_{\Gi} \left(\hu \cdot \tu   \right)\left(\tv   \cdot n\right) 
  =-\int_{\Gi}{\left(\hu \cdot \tau\right)\left(\tv   \cdot n\right)\left(\tv   \cdot \tau\right)} .
\]
There exists a vector field $T:\overline{\Omega}\longrightarrow \R^2$ which is $C^2$, equal to $0$ outside a given neighborhood of $\Gi$
and verifies $T_{|\Gi}= \tau$. 
We apply Lemma \ref{l:lemast} with 
$(\tv   ,\tv   ,\left(\hu \cdot T\right)T)$
instead of $(u,v,w)$, observing that $\tv$ 
is divergence free since the function $\tilde{\varphi}$ is harmonic. 
We arrive at
\begin{align*}
2\int_{\Gi}{\left(\hu \cdot \tau\right)\left(\tv   \cdot n\right)\left(\tv   \cdot \tau\right)}
    &= 2\int_{\Gi}{\left(\tv   \cdot n\right)\left(\tv   \cdot \left(\hu \cdot T\right)T\right)} \\
    &= - \int_{\Omega}{|\tv   |^2\,\div(\left(\hu \cdot T\right)T)} 
      + 2\int_{\Omega}{\tv   \cdot\left(\left(\tv   \cdot\nabla\right)\left(\hu \cdot T\right)T\right)} .
\end{align*}
By  H\"older's inequality we have that
\begin{equation}
\left|\int_{\Gi}{\left(\hu \cdot \tau\right)\left(\tv   \cdot n\right)\left(\tv   \cdot t\right)}\right|
    \leq C \Vert \tv   \Vert_{L^{\frac{2p}{p-1}} (\Omega)}^2 \Vert \hu \Vert_{W^{1,p}  (\Omega)} .
\end{equation}
To conclude, we use the interpolation inequality  \eqref{interpo} 
and the estimates of Proposition \ref{p:yudo-ell-bound} and of Proposition \ref{p:aux-test-reg}.
\end{proof}

Regarding the second to last term in the right hand side of \eqref{auxe},
we first establish the following result on the circulations  $\tilde{\mathcal{D}}_i$ of $\tv$ around each $\Gamma_i$.

\begin{lem}\label{l:ineq-PsiD}
For all $\varepsilon>0$, there exists $M(\varepsilon)>0$ such that  for every $0\leq t_0<t_1 \leq T$,
\begin{align}
\nonumber
\left|\int_{t_0}^{t_1}{\sum_{i\in I^*}{\psi_i'\tilde{\mathcal{D}}_i}}\right|
    &\leq M(\varepsilon)\left( (t_1-t_0)  \Vert \tomein\Vert_{L^{\infty}([0,t];L^{\infty}(\Gi,|g|))}^2
           + (t_1-t_0)\sum_{i\in I_{\mathrm{in}}}{ |\tilde{\mathcal{C}}_{i,0}|^2}
           +\int_{t_0}^{t_1}{\Vert\tv\Vert_{L^{2}(\Omega)}^2}  \right) \\
\label{evla3}
     &+ \varepsilon  \int_{t_0}^{t_1}{\left(   
           \Vert  \tu  \Vert_{L^{2}(\Omega)}^{2}
           +\Vert  \tu  \Vert_{L^2\left(\Gamma, |g|\right)}^{2}    \right)}.
\end{align}
\end{lem}

\begin{proof}
On the one hand, for $i\in I_{\mathrm{in}}$, by \eqref{rgf} and \eqref{noel}, we obtain that for any $t \geq 0$, 
\begin{equation}
\label{EQA}
|\tilde{\mathcal{D}}_i (t)|
    \leq  |\tilde{\mathcal{C}}_{i,0} |
    + \int_{0}^{t}\int_{\Gamma_i}{{|\tomein||g|}} 
    \leq  |\tilde{\mathcal{C}}_{i,0} |
    + T 
     \Vert \tomein\Vert_{L^{\infty}([0,t];L^{\infty}(\Gi,|g|))} .
\end{equation}
On the other hand, 
 by the Cauchy-Schwarz inequality and 
\eqref{hoapplied}
 we obtain that  there exists a  constant $C>0$ such that for any $t \geq 0$, for all $i\in I_{out}$,
\begin{equation}
  \label{EQB}
  \left| \tilde{\mathcal{D}}_i(t) \right|\leq C\Vert \tv   (t,.) \Vert_{L^2 (\Omega)}.
\end{equation}
Combining   \eqref{EQA} and   \eqref{EQB},  
 we get that there exists a constant $C>0$ such that  for every $0\leq t_0<t_1 \leq T$,
\begin{equation}\label{EQAB}
\int_{t_0}^{t_1}{\sum_{i\in I^*}{|\tilde{\mathcal{D}}_i|^2}}
    \leq C\int_{t_0}^{t_1}{\Vert \tv\Vert_{L^2(\Omega)}^2}
     + C(t_1-t_0)\Vert \tomein\Vert_{L^{\infty}([0,t];L^{\infty}(\Gi,|g|))}^2
     + C(t_1-t_0)\sum_{i\in I_{\mathrm{in}}}{ |\tilde{\mathcal{C}}_{i,0}|^2} .
\end{equation}
Moreover, for any $\varepsilon>0$, there exists $M(\varepsilon)>0$ such that  for every $0\leq t_0<t_1 \leq T$,
\begin{equation}\label{EQC}
\left|\int_{t_0}^{t_1}{\sum_{i\in I^*}{\psi_i'\tilde{\mathcal{D}}_i}}\right|
    \leq M(\varepsilon) \int_{t_0}^{t_1}{\sum_{i\in I^*}{|\tilde{\mathcal{D}}_i|^2}} 
    + \varepsilon C^{-1} \int_{t_0}^{t_1}{\sum_{i\in I^*}{|\psi_i'|^2}}     ,
\end{equation}
 where $C>0$ is the constant which appears in \eqref{desire}.
By combining \eqref{EQC}, \eqref{EQAB} and \eqref{desire}   we obtain 
\eqref{evla3}.
\end{proof}

Regarding the last term in the right hand side of \eqref{auxe},
we have the following estimate. 

\begin{lem}\label{l:tdb-tomein}
There exists a constant $C>0$ such that for any $0<t_0<t_1<T$:
\begin{equation}
\left|\int_{t_0}^{t_1}{\int_{\Gi}{\tilde{\varphi} \tomein  g}}\right|
    \leq C\int_{t_0}^{t_1}{\Vert\tv\Vert_{L^{2}(\Omega)}^2}
       +(t_1-t_0)\Vert\tomein\Vert_{L^{\infty}([0,T]\times \Gi)}^2.
\end{equation}
\end{lem}

\begin{proof}
By classical trace theory there exists $C>0$ such that 
 \begin{align*}
 \Vert\tilde{\varphi}\Vert_{L^2(\Gamma)} \leq C \Vert\tilde{\varphi}\Vert_{H^1(\Omega)}
 %\\ 
 \leq C' \Vert\tilde{v}\Vert_{L^2(\Omega)} ,
 \end{align*}
 for another $C' >0$ by Poincar\'e inequality, using that $\tilde{\varphi}$ is equal to $0$ on $\Go$.
Therefore Lemma \ref{l:tdb-tomein} follows by using the Cauchy-Schwarz inequality.
\end{proof}

Finally to bound the fourth term in the right hand side of \eqref{auxe}, we proceed as in the proof of Corollary \ref{t:energ-eq-cor}. More precisely 
 we first use H\"older's inequality to get that 
 for all $p$ in $(1,+\infty)$, 
 \begin{align*} 
\left|\int_{\Omega}
\Big( {\homega  \tu  \cdot \tv   ^{\perp}}
+ {\tv   \cdot \left(\left( \tu  \cdot \nabla\right)\hu \right)}
+ { \tu  \cdot \left(\left(\tv   \cdot \nabla\right)\hu \right)} \Big)
\right|
    \leq \left(\Vert \tv   \Vert_{L^{\frac{2p}{p-1}}(\Omega)}^2+\Vert  \tu  \Vert_{L^{\frac{2p}{p-1}}(\Omega)}^2\right)\Vert \hu \Vert_{W^{1,p}(\Omega)} .
\end{align*}
Then, we observe that both $\tu  $ and $\tv   $ are in $L^{\infty}([0,T];L^{\infty}(\Omega))$, respectively thanks to Proposition \ref{prop-f}
and  Proposition \ref{p:aux-test-reg},  the Sobolev embedding theorem and \eqref{e:u-nab-perp-psi} and  \eqref{hatv}.
  Thus, by  the interpolation inequality \eqref{interpo}  and by Proposition \ref{p:yudo-ell-bound}, 
  we  deduce  that  there exists a positive constant $C$ such that   for all $p$ in $[2,+\infty)$, 
 \begin{align} 
\left|\int_{\Omega}
\Big( {\homega  \tu  \cdot \tv   ^{\perp}}
+ {\tv   \cdot \left(\left( \tu  \cdot \nabla\right)\hu \right)}
+ { \tu  \cdot \left(\left(\tv   \cdot \nabla\right)\hu \right)} \Big)
\right|
   &\leq   C p\left(\Vert \tu  \Vert_{L^2 (\Omega)}^{2\left(1-\frac{1}{p}\right)}+\Vert \tv   \Vert_{L^2 (\Omega)}^{2\left(1-\frac{1}{p}\right)}\right). \label{evla4}
\end{align}

Therefore Corollary  \ref{auxi-cor} is a consequence of Lemma \ref{l:ineq-go}; Lemma \ref{l:tdb-tomein},
Lemma \ref{l:ineq-gi}, and Lemma \ref{l:ineq-PsiD} with $\varepsilon=\frac{1}{8}$ 
and  \eqref{evla4}.
\end{proof}

%%%%%%%%%%%%%%%%%%%%%%%%%%%%%%%%%%%%%%%%%%%%%%%%%%%%%%%%%%%%
\section{Osgood argument and end of the proof}
\label{sec-game}

In this section we combine the two energy estimates respectively obtained in  
Section \ref{p:ener-eq} and in  Section \ref{sec-aux-esti} and use an Osgood argument to conclude the proof of Theorem \ref{main-thm}. 
Indeed,  by summing \eqref{i:ener-ineq} and  \eqref{i:aux-ineq}, we obtain that there exists a constant $C>0$ such that, for  $0\leq t_0<t_1 \leq t \leq T$,  for every $p \geq 2$, 
\begin{align}\label{tag}
&\left[\Vert  \tu   \Vert_{L^2(\Omega)}^2+\Vert \tv    \Vert_{L^2(\Omega)}^2\right]_{t_0}^{t_1}
     + \frac{1}{2} \int_{t_0}^{t_1}{\int_{\Gamma}{| \tu  |^2| g|}} \\  \nonumber
  &\,   \leq  C \int_{t_0}^{t_1} \left(\Vert \tv    \Vert_{L^2(\Omega)}^2
                   + \Vert \tu \Vert_{L^2(\Omega)}^2
                   + p \Vert \tv \Vert_{L^{2}(\Omega)}^{2\frac{p-1}{p}} 
                   + p \Vert \tu \Vert_{L^{2}(\Omega)}^{2\frac{p-1}{p}}    \right)
           \\  \nonumber    &\quad    + C(t_1-t_0) \Big( \Vert \tomein \Vert_{L^{\infty}([0,t];L^{\infty}(\Gi,|g|))}^2
                   +   \sum_{i\in I_{\mathrm{in}}}{ |\tilde{\mathcal{C}}_{i,0}|^2}  \Big).
\end{align}
We first omit  the second term on the left hand side of \eqref{tag} and focus on the estimate of the time-dependent function
\begin{equation} \label{gam}
z := \Vert  \tu  \Vert_{L^2(\Omega)}^2+\Vert \tv   \Vert_{L^2(\Omega)}^2 . 
\end{equation} 
We obtain that for $0 \leq t_0<t_1 \leq t \leq T$ and for every $p \geq 2$,
\begin{equation}\label{i:grontyp3}
[z]_{t_0}^{t_1}
    \leq C\int_{t_0}^{t_1}{\left(z(t)+pz(t)^{1-\frac{1}{p}}\right)\mathrm{dt}} 
    + C(t_1-t_0) \Big(\Vert\tomein\Vert_{L^{\infty}([0,t];L^{\infty}(\Gi,|g|))}^2
    +   \sum_{i\in I_{\mathrm{in}}}{ |\tilde{\mathcal{C}}_{i,0}|^2} \Big).
\end{equation}
In the case where $z(t)$ is more than one, the term in the first parenthesis above can be bounded by $(1+p) z(t)$, and a classical Gronwall argument can be applied. 
The case where $z(t)$ is less than one, which is of particular interest in view of the uniqueness issue, requires to replace  the Gronwall argument by  Osgood's lemma,  
  as this was done by Yudovich in \cite{yudo2005,yudo1995} in the impermeable case. 
To this end  we first establish the following result in view of 
 minimizing the right hand side above 
with respect to $p$ locally in time. 
\begin{lem}\label{l:techn}
Let $a\geq 0$   and $y$ a non negative continuous and non-decreasing function.
Let $F$ a non negative continuous function on $[2,+\infty) \times \R_+$ which is  increasing with respect to the second variable. 
We assume that for every $0<t_0<t_1<T$, and for every $p\in [2,+\infty)$, 
\begin{equation}\label{i:grontyp-y1}
 [y]_{t_0}^{t_1}  \leq a (t_1-t_0) + \int_{t_0}^{t_1}{F(p,y(s))\mathrm{ds}}.
\end{equation}
Let $\mu$ a continuous function from $\R_+$ to $\R_+$. 
We assume that for any $x\geq 0$, there exists $p_x$ in  $[2,+\infty)$ such that 
\begin{equation}
\label{i:def-px}
F(p_x,x) \leq  \mu(x).
\end{equation}
Then, for all $0<t<T$,
\begin{equation}\label{i:grontyp-y2}
[y]_{0}^{t} \leq  at + \int_{0}^{t}{\mu(y(s))\mathrm{ds}}.
\end{equation}
\end{lem}
\begin{proof}[Proof of Lemma \ref{l:techn}]
Let $t$ in $(0,T)$ and $n$ in $\N^*$.
Using \eqref{i:grontyp-y1} and \eqref{i:def-px}, we get that there are some $(p_{y(\frac{(k+1)t}{n})})_{0 \leq k \leq n-1}$ in $[2,+\infty)^n$  and $C>0$ such that 
\begin{align}
[y]_{0}^{t}
    = \sum_{k=0}^{n-1} \Big( y\left(\frac{(k+1)t}{n}\right)-y\left(\frac{kt}{n}\right) \Big) %\\
\label{i:tech-prf1}
 %   &
 \leq
     at + \sum_{k=0}^{n-1}{\int_{\frac{kt}{n}}^{\frac{(k+1)t}{n}}{F\left( p_{y\left(\frac{(k+1)t}{n}\right)},y(s)\right)\mathrm{ds}}} ,
\end{align}
and for $0 \leq k \leq n-1$, 
\begin{equation}\label{i:tech-prf3}
F\left(p_{y\left(\frac{(k+1)t}{n}\right)},y\left(\frac{(k+1)t}{n}\right)\right)
    \leq C \mu\left(y\left(\frac{(k+1)t}{n}\right)\right).
\end{equation}
Then, as $y$ is an non decreasing function of $t$ and $F$ is an increasing function of $y$, the function $t\mapsto F(p,y(t))$ is increasing, which leads to 
\begin{equation}\label{i:tech-prf2}
\int_{\frac{kt}{n}}^{\frac{(k+1)t}{n}}{F\left(p_{y\left(\frac{(k+1)t}{n}\right)},y(s)\right)\mathrm{ds}}
    \leq \frac{t}{n} F\left((p_{y\left(\frac{(k+1)t}{n}\right)},y\left(\frac{(k+1)t}{n}\right)\right).
\end{equation}

Combining \eqref{i:tech-prf1}, \eqref{i:tech-prf3} and \eqref{i:tech-prf2}, we obtain that
\begin{equation}
[y]_{0}^{t}
    \leq at + \frac{Ct}{n}  \sum_{k=0}^{n-1}{\mu\left(y\left(\frac{(k+1)t}{n}\right)\right)}.
\end{equation}
As this is true for every $n$, we conclude by Riemann summations, using the fact that $\mu$ is continuous.
\end{proof}
We define the function $y$ by 
\begin{equation}  \label{blacO}
y(t):=\underset{s\in[0,t]}{\mathrm{max}}\,{z(s)},
\end{equation}
 where $z$ is the function defined in  \eqref{gam}, and 
 we set 
\begin{equation} \label{blac}
F(p,x):=  x + px^{1-\frac{1}{p}}, \, 
\mu(x):=Cx( 1 + |\ln(x)| ) 
\end{equation}
and
\begin{equation}
a := C( \Vert\tomein\Vert_{L^{\infty}([0,t];L^{\infty}(\Gi,|g|))}^2
    +   \sum_{i\in I_{\mathrm{in}}}{ |\tilde{\mathcal{C}}_{i,0}|^2}) ,
\end{equation}
where $C$ is the constant which appears in 
\eqref{i:grontyp3}. 
The  condition \eqref{i:def-px} is verified with
$p_x =  |\ln(x)| $ for $x$ in $\left[0,e^{-2}\right]$ and 
 $p_x = 2$ for $x$ in $(e^{-2},+\infty)$.
Moreover, according to  \eqref{i:grontyp3} 
  the function $y$ verifies \eqref{i:grontyp-y1}.
Therefore, by Lemma \ref{l:techn},  the function $y$ verifies \eqref{i:grontyp-y2} for all $t\in [0,T]$. 

Let us now recall Osgood's lemma,  see for example  \cite[Lemma 3.4.]{BCD}.
\begin{lem} \label{t:osgood}
Let $\mu:[0,1]\longrightarrow [0,+\infty)$  an increasing continuous function with 
$\mu(0) = 0$, $T>0$ and $c \geq 0$. 
Let  $y: \R_+ \longrightarrow [0,1]$  a  function verifying, for $t $ in $[ 0,T]$, 
\begin{equation}
\label{gro}
y(t) \leq c + \int_{0}^{t}{\mu(y(s))ds},
\end{equation}
In the case where $c >0$, then, for $t $ in $[ 0,T]$, 
$$ \int_{c }^{y(t)}{\frac{dx}{\mu(x)}} \leq  t.$$ 
In the case where $c =0$, if the function $\mu$ also satisfies 
\begin{equation*}
\int_{0}^{1}{\frac{dt}{\mu(t)}} = +\infty ,
\end{equation*}
then  for $t $ in $[ 0,T]$,   $y(t) = 0 $. 
\end{lem} \bigskip

We now apply Osgood's Lemma to the function $y$ defined in \eqref{blacO}, with $\mu$ and $a$ defined in  \eqref{blac}, and  $c=  y(0) + a T$.
Focusing on the first case 
 this yields  that  for $0 \leq t  \leq T$, 
$$ \int_{  y(0) + a T}^{y(t)}{\frac{dx}{\mu(x)}} \leq  t, $$ 
and thus in particular that  for any $t  \geq 0$,
$$ \int_{  y(0) + a t }^{y(t)}{\frac{dx}{\mu(x)}} \leq  t ,$$ 
which leads to 
\begin{equation} \label{top}
y(t) \leq e \,  (y(0)+ta)^{e^{-Ct}}.   
\end{equation}
Using that, similarly to \eqref{lactou}, there holds  
\begin{equation}
\Vert\tv (0,\cdot) \Vert_{L^2(\Omega)} 
    \leq C \big( \Vert\tu (0,\cdot)  \Vert_{L^2(\Omega)}
     + \Vert\tomega (0,\cdot)  \Vert_{L^2(\Omega)} \big) ,
\end{equation} 
and recalling 
\eqref{gam}  and  \eqref{blacO}, 
we  deduce from \eqref{top} that 
\begin{equation*}
\Vert \tu(t)\Vert_{L^2(\Omega)}^2
    \leq C\left( \Vert\tu (0,\cdot)  \Vert_{L^2(\Omega)}^2
                    + \Vert\tomega (0,\cdot)  \Vert_{L^2(\Omega)}^2 
               +t\Big(\Vert\tomein\Vert_{L^{\infty}([0,t];L^{\infty}(\Gi,|g|))}^2
                    +\sum_{i\in I_{\mathrm{in}}}{ |\tilde{\mathcal{C}}_{i,0}|^2} \Big)\right)^{e^{-Ct}}.
\end{equation*}
By using Proposition \ref{p:yudo-ell-bound}
we conclude that there exists a continuous function $F: (\R_+)^3  \times \R^{I^*} \mapsto \R_+$ satisfying $F(\tau,0,0,0)=0$ for all $\tau \geq 0$, and such that for every $t\in [0,T]$, 
\begin{gather*}
\Vert \tu(t)\Vert_{L^2(\Omega)}^2
    \leq F\big(t,
    \Vert\omega^1_0-\omega^2_0\Vert_{L^{\infty}(\Omega)} ,
    \Vert\omein^1-\omein^2\Vert_{L^{\infty}([0,t];L^{\infty}(\Gi,|g|))},
    (\mathcal{C}^1_{i,0} - \mathcal{C}^2_{i,0})_{i\in I^*} \big) .
\end{gather*}
Moreover going back to  \eqref{tag}, we  deduce a similar bound of the second term on the left  hand side and we arrive at 
\eqref{stab-est}. 
 The uniqueness result  corresponds to the case where  $y(0)$ and $ a$ are both zero for which we apply the second case of Osgood's Lemma.

\begin{rem}
With some bookkeeping we observe that formally the inequality \eqref{tag} is obtained 
by applying the weak formulation \eqref{e:weak-trans-diff} with a test function which is a combination of 
 $\tilde \psi $, of $\tilde{\varphi}$ and of  $ \tilde{\mathcal{D}}_i\,  g^i $, for $i$ in  $I^*$, and that such a test function is a non-local
 operator of order $0$ acting on $\tilde \psi$.  This is reminiscent of the Kreiss symmetrizer technics in hyperbolic theory, see for example \cite{BGS,KL}.
\end{rem}

%%%%%%%
\ \par \ 

\paragraph{\textbf{Acknowledgements.}}
The authors are partially supported by the Agence Nationale de la Recherche, Project IFSMACS, grant ANR-15-CE40-0010, Project SINGFLOWS, ANR-18-CE40-0027-01, Project BORDS, grant ANR-16-CE40-0027-01 and the H2020-MSCA-ITN-2017 program, Project ConFlex, Grant ETN-765579.
The last author warmly thank Maria Kazakova, Gennady Alekseev and Alexander Mamontov for their kind help regarding the russian literature on the subject, and David Lannes for interesting discussions on the subject.
This work was partly accomplished while  F.S. was participating in a program hosted by the Mathematical Sciences Research Institute in Berkeley, California, during the Spring 2021 semester, and supported by the National Science Foundation under Grant No. DMS-1928930.

\end{document}